\documentclass[a4paper,11pt]{article}
\usepackage{authblk}
\usepackage{cancel}
\usepackage[nointlimits]{amsmath}
\usepackage{amsfonts}
\usepackage{amssymb}
\usepackage{amsmath}
\usepackage{amsthm}
\usepackage{mathrsfs}
\usepackage{mathtools}
\usepackage{hyperref}
\usepackage[capitalise]{cleveref}
\usepackage{latexsym}
\usepackage{graphicx}
\usepackage{layout}
\usepackage{fancyhdr}

\usepackage{lscape} 

\usepackage{amscd}
\usepackage{color}

\newtheorem{lemma}     {Lemma}[section]

\newtheorem{teorema1}   [lemma]{Theorem}
\newtheorem{prop}      [lemma]{Proposition}

\newtheorem{assum} [lemma]{Assumption}
\newtheorem{cong1}      [lemma]{Conjecture}
\newtheorem{remark1}    [lemma]{Remark}
\newtheorem{defin}     [lemma]{Definition}

\newenvironment{Theorem}[1][]
{\begin{teorema1}[#1]\begin{samepage}}{\end{samepage}\end{teorema1}}

\newenvironment{Remark}[1][]
   {\begin{remark1}[#1]\begin{samepage}}{\end{samepage}\end{remark1}}

\numberwithin{equation}{section}

\renewcommand{\a}{\alpha}

\newcommand{\dis}{\displaystyle}
\renewcommand{\L}{\mathcal{L}}

\newcommand{\mmmintone}[1]{{\dis{\int\kern -.43cm
-}}_{\kern-.21cm\substack{#1}}\;\;}
\newcommand{\mmmintwo}[2]{{\dis{\int\kern -.43cm
-}}_{\kern-.21cm\substack{#1}}^{\substack{#2}}\;\;}
\newcommand{\submint}{{\scriptstyle{\int\kern -.66em -}}}
\newcommand{\submintone}[1]{{\scriptstyle{\int\kern -.66em
-}}_{\scriptscriptstyle{\kern-.21em\substack{#1}}}}
\newcommand{\fracmint}{{\textstyle{\int\kern -.88em -}}}
\newcommand{\fracmintone}[1]{{\textstyle{\int\kern -.88em
-}}_{\scriptscriptstyle{\kern-.21em\substack{#1}}}\;}

\newcommand{\nada}[1]{}

\renewcommand{\L}{\mathcal{L}}

\DeclareFontFamily{U}{BOONDOX-calo}{\skewchar\font=45 }
\DeclareFontShape{U}{BOONDOX-calo}{m}{n}{
<-> s*[1.05] BOONDOX-r-calo}{}
\DeclareFontShape{U}{BOONDOX-calo}{b}{n}{
<-> s*[1.05] BOONDOX-b-calo}{}
\DeclareMathAlphabet{\mathcalboondox}{U}{BOONDOX-calo}{m}{n}
\SetMathAlphabet{\mathcalboondox}{bold}{U}{BOONDOX-calo}{b}{n}
\DeclareMathAlphabet{\mathbcalboondox}{U}{BOONDOX-calo}{b}{n}
\newcommand{\mcb}[1]{{\mathcalboondox #1}}

\usepackage{bbm}

\usepackage{setspace}
\usepackage{mathrsfs} 
\usepackage{amsmath}
\usepackage{graphicx}
\usepackage{multirow}
\usepackage{eurosym}
\usepackage{amstext}
\usepackage{hyperref}
\usepackage{grffile} 
\usepackage{color}
\usepackage{mathabx}
\usepackage{float}
\usepackage{pdfpages}
\usepackage{comment}
\usepackage{pifont}
\usepackage{multirow}
\usepackage[dvipsnames]{xcolor}

\usepackage[linewidth=1pt]{mdframed}
\usepackage{lipsum}

\usepackage[T1]{fontenc}
\usepackage[utf8]{inputenc} 
\usepackage{lmodern}

\usepackage[normalem]{ulem}

\title{Hydrodynamic limit for some gradient and attractive spin models}
\date{}

\author[1]{Chiara Franceschini}
\author[2]{Patr\'{i}cia Gon\c{c}alves}
\author[3]{Kohei Hayashi}
\author[4]{Makiko Sasada}
\affil[1]{Dipartimento di Scienze Fisiche, Informatiche e Matematiche, Universit\`{a} degli Studi di Modena e
Reggio Emilia}
\affil[2]{Departamento de Matem\'atica, Instituto Superior T\'ecnico, Universidade de Lisboa}
\affil[3]{Department of Mathematics, Graduate School of Science, The University of Osaka}
\affil[3]{RIKEN Center for Interdisciplinary Theoretical and Mathematical Sciences} 
\affil[4]{Graduate School of Mathematical Sciences, The University of Tokyo}

\begin{document}

\maketitle
    
\begin{abstract}
We study the hydrodynamic limit for three gradient spin models: generalized Kipnis-Marchioro-Presutti (KMP), its discrete version and a family of harmonic models, under symmetric and nearest-neighbor interactions. These three models share some universal properties: occupation variables are unbounded, all these processes are of gradient type, their invariant measures are product with spatially homogeneous weights, and, notably, they are all attractive, meaning that the process preserves the partial order of measures along the dynamics. In view of hydrodynamics of large-scale interacting systems, dealing with processes taking values in unbounded configuration spaces is known to be a challenging problem. In the present paper, we show the hydrodynamic limit for all three models listed above in a comprehensive way, and show as a main result, that, under the diffusive time scaling, the hydrodynamic equation is given by the heat equation with model-dependent diffusion coefficient. Our novelty is showing the attractiveness for each model, which is crucial for the proof of hydrodynamics. 
\end{abstract}

\section{Introduction}
In this article, we consider different microscopic spin models that share some universal properties that allow us to derive in a similar fashion the hydrodynamic limit for all of them. These models are energy models, but also particle models, which we describe as follows.

\subsection{The microscopic models}
The first model we analyze is the  Kipnis, Marchioro and Presutti (KMP) introduced in 1982 as a stochastic model to study heat conductivity in a one-dimensional lattice system \cite{kipnis1982heat}. 
In this article, we study a generalized version of the KMP model proposed in \cite{carinci2013duality} where energy is redistributed according to a Gamma distribution of parameters $(2\mathfrak s,2\mathfrak s)$, with $\mathfrak s>0$. 
In their original paper \cite{kipnis1982heat}, they also introduced, via a duality relation, a different model of interacting discrete particles, that we refer to as discrete KMP. Here particles are redistributed between two neighboring sites according to a discrete uniform distribution. Lastly, we also consider a family of models recently introduced in \cite{frassek2020non}, which are referred to as Harmonic models. This family, also labeled by spin $\mathfrak s$, is different from the KMP-type models described above since particles perform zero-range jumps, namely the rate only depends on the occupation variable on the departure site.  However, unlike classical zero-range processes, more than one particle may jump at a time from a site to another one.
We note that all the models considered here have nearest-neighbor interactions, they are  of gradient type and with random symmetric jump probability rates. 
The gradient property of the models translates by saying that the instantaneous current of the systems writes as the gradient of some local function of the dynamics.  However, it is clear that the microscopic behavior of these models posses different features, whereas the proof we present, which does not depend on the specific type of model, is quite general.

\subsection{The microscopic space}
For simplicity of the presentation, we study the models in a simple setting, a discrete one-dimensional torus with periodic boundary conditions with $N$ sites. In fact, our proof applied to any $d$-dimensional torus, and we leave the details of this extension to the interested reader. The fact of working under periodic boundary conditions allows to have several powerful features which we now describe. Nevertheless, we leave as a future work the study of these models in a boundary-driven context as well as under long range interactions, in which case the hydrodynamic equations are of fractional type.

\subsection{Properties of the models}
In our setup, one can easily check that the invariant measure in all the models is a homogeneous product measure whose marginals are given explicitly below. Moreover, the dynamics conserves the quantity of interest of each model: either the total energy for continuous KMP models, or the total number of particles for discrete KMP and Harmonic models. Our proof for the hydrodynamic limit relies on this conservation law and on attractiveness, a powerful property to study Markov processes, see \cite[Section 11.2]{liggett1985interacting}. Attractiveness guarantees a sort of monotonicity in the context of particle systems, where the partial order of configurations is preserved over time.
On the other hand, unlike exclusion-type processes, the invariant measures do not have super-exponential tails and the occupation variables are unbounded. This unboundedness is usually quite difficult to overcome for the proof of scaling limits of interacting particle systems, and here we handle it by using heavily the conservation law and the attractiveness property present in all the models. The drawback is to force the models to initiate from measures that are stochastically dominated by the invariant measure; see Assumption \ref{assum:initial_measure}.

\subsection{The hydrodynamic limit}
Since its development in the 90's the theory of hydrodynamic limit for particle systems is quite rich, we mention below some works with a similar statement to ours. In \cite{rezakhanlou1991hydrodynamic} Rezakhanlou proved the hydrodynamic limit for exclusion and zero-range type processes on infinite volume relying on attractiveness. 
In \cite{suzuki1993hydrodynamic}, hydrodynamics is proved for an energy model with a similar redistribution rule of KMP, see also \cite{feng1997microscopic}.
In particular, their redistribution rule coincides with that of Section \ref{cts model} when $D=4\mathfrak s$ in their notation. However, the rates for energy redistribution are different and they found a diffusion equation that reduces to the porous medium equation when $ D=4\mathfrak s$. 
The symmetric inclusion process also belongs to the same universality class, but it is not attractive and its hydrodynamic limit was studied in \cite{franceschini2022symmetric} for an open finite chain. Additionally, the proof in that case heavily relies on the dual purely absorbing process.

\subsection{Our contribution}
In  a nutshell our contribution is this article is twofold. First, we prove that all our models satisfy the gradient condition. As we mentioned above this means that the current of the systems writes as the gradient of some local function which in all our models is written in terms of the energy or particle configuration. This gives already an intuition that the hydrodynamic equation is indeed the heat equation. 
Second, we also prove that all our models are attractive. Depending on the model, this is done in two different ways: either by showing that the basic coupling preserves the monotonicity of the model or by employing a criterion (Theorem 2.9 of \cite{gobron2010couplings}) on the jump rates that immediately implies that monotonicity property. At this level we need to make a restriction on the spin $\mathfrak s$ of the Harmonic models, since we can prove that they are attractive under the restriction $\mathfrak s\ge 1/2$ and we leave the proof in the remaining cases open. 
With these two ingredients in hands, we apply the powerful entropy method developed by \cite{guo1988nonlinear} but we assume that our starting measures are stochastically dominated by an invariant measure (that we specify below for each model).

\subsection{Future problems}
We leave as a future problem the derivation of the hydrodynamic limit for all the models starting from more general measures, i.e. without assuming the attractiveness nor the stochastic domination of the initial measures. A natural question that follows is the analysis of the non-equilibrium fluctuations for these models, i.e. the description of the  stochastic fluctuations of the random microscopic dynamics around the hydrodynamic behavior. All these questions can be posed when the system is boundary driven for which the proof will technically be much more involved, since the conservation law is destroyed.  Another problem is the derivation of the {dynamical} large deviations principle for all these models. Since our proof is based on the entropy method, we have now paved the way to natural look at our results at the level of large deviations. 
Here follows and outline of the article.

\subsection{Outline of the article}
The rest of the paper is organized as follows: in Section \ref{models} we introduced the three families of models and in Theorem \ref{HDL} we state the hydrodynamic limit, our main result. 
We dedicate Section \ref{attra} to show the attractiveness of the models. Lastly, in Section \ref{prooof} we prove our main result; first, by showing tightness of the process of empirical measures and then we characterize the limit points of the tight sequence and show that such limit is concentrated on trajectories of measures that are absolute continuous with respect to the Lebesgue measure and whose density is the weak solution of the heat equation with a diffusion coefficient that depends on the chosen model.

\section{The models and statement of the result}
\label{models}
{Throughout the present paper, set $\mathbb N=\{1,2,\ldots,\}$ and $\mathbb N_0=\{0,1,\ldots\}$. }
Let $N \in \mathbb{N}$ and denote by $\mathbb{T}_{N} = \mathbb R/\mathbb Z \cong \lbrace 0, 1, \ldots, N-1 \rbrace$ the discrete one-dimensional torus with $N$ sites and by $\mathbb{T} \cong [0,1)$ the corresponding macroscopic one-dimensional torus. The energy models of interest have state space given by $ \mathbb{R}_+^{\mathbb{T}_N}$, of which elements are configurations of energy; whereas the configuration space of particle models is given by $\mathbb{N}_0^{\mathbb{T}_N}$ whose elements are referred to as particle configurations. 
Elements of each model are denoted by Greek letters $\eta$.
In the first case, the quantity of energy at site $x \in \mathbb{T}_{N}$ is denoted by $\eta_x \in \mathbb{R}_+$, whereas for the letter $\eta_x \in \mathbb{N}_0$ represents the number of particles at site $x \in \mathbb{T}_{N}$.
In what follows, we describe the dynamics of three spin models that share similar features. In particular, in all these models there is a quantity which is conserved by the dynamics, either the total energy or the total number of particles. Moreover, for all these models there is a one-parameter family of product probability measures, which are reversible, where the parameter represents, up to a constant, the average energy or particle density per site.

\subsection{Energy models} \label{cts model}
We study a family of models labeled by the parameter $\mathfrak s>0$ where energy is exchanged in a random way among nearest-neighbor sites. 

\subsubsection{Generalized Kipnis-Marchioro-Presutti model}
Let us consider a Markov process on the configuration space $\Omega_N^{\mathrm{gKMP}}=\mathbb R_+^{\mathbb T_N}$. 
The infinitesimal evolution of the process is encoded in the Markov generator acting of functions $f: \Omega_N^{\mathrm{gKMP}} \to \mathbb{R}$ as
\begin{equation}
\label{gkmpgen}
L_N^{\mathrm{gKMP}} f (\eta)
= \sum_{x \in \mathbb{T}_N} L_{x,x+1}^{\mathrm{gKMP}} f (\eta) 
\end{equation}
where we define
\begin{equation*}
L_{x,x+1}^{\mathrm{gKMP}} f (\eta) 
= \int_0^1  \gamma_{\mathfrak s}(u)  
\big[f(\eta^{x,x+1,u}) -f(\eta)\big] du.
\end{equation*}
The above dynamics corresponds to redistributing energy between two nearest neighboring sites with rate $1$ so that the ratio of the new energy follows a Beta distribution with both parameters $2\mathfrak s > 0$, i.e., 
\begin{equation}
\label{eq:dkmp_redistribution_weight}
\gamma_{\mathfrak s}(u) = \frac{\Gamma(4\mathfrak s)}{\Gamma(2\mathfrak s)^2} u^{2\mathfrak s-1} (1-u)^{2\mathfrak s-1}
\end{equation}
where $\Gamma(z)=\int_0^\infty y^{z-1}e^{-y}dy$ denotes the Gamma function, and for $u\in[0,1]$  and $\eta \in\Omega_N^{\textrm{gKMP}}$ the new configuration is defined by 
\begin{eqnarray*}
(\eta^{x,y,u})_z = 
\begin{cases}
\begin{aligned}
& u(\eta_x + \eta_y)  && \text{ if } z=x, \\
& (1-u)(\eta_x + \eta_y)  && \text{ if } z=y, \\
& \eta_z  && \text{ otherwise. }  \\
\end{aligned}
\end{cases}
\end{eqnarray*}
It will be convenient to split the action of the generator as
$L_{x,x+1}^{\mathrm{gKMP}} = \nabla_{x,x+1}^{\mathrm{gKMP}} + \nabla_{x+1,x}^{\mathrm{gKMP}}$, where $\nabla_{x,x+1}^{\mathrm{gKMP}}$ is the energy transfer from site $x$ to site $x+1$ and $\nabla_{x+1,x}^{\mathrm{gKMP}}$ is that from site $x+1$ to site $x$: 
\begin{align*}
\nabla_{x,x+1}^{\mathrm{gKMP}} f(\eta)&= \frac{\Gamma(4\mathfrak s)}{\Gamma(2\mathfrak s)^2} \int_{0}^{\eta_x}   \dfrac{(\eta_x - \alpha)^{2\mathfrak s-1}(\eta_{x+1} + \alpha)^{2\mathfrak s-1}}{(\eta_x + \eta_{x+1})^{4\mathfrak s-1}} \left[ f(\eta - \alpha \delta_x + \alpha \delta_{x+1} ) - f(\eta) \right] d\alpha \\
\nabla_{x+1,x}^{\mathrm{gKMP}} f(\eta) 
&=  \frac{\Gamma(4\mathfrak s)}{\Gamma(2\mathfrak s)^2} \int_{0}^{\eta_{x+1}}  \dfrac{(\eta_x + \alpha)^{2\mathfrak s-1}(\eta_{x+1} - \alpha)^{2\mathfrak s-1}}{(\eta_x + \eta_{x+1})^{4\mathfrak s-1}} \left[ f(\eta + \alpha \delta_x - \alpha \delta_{x+1} ) - f(\eta) \right] d\alpha
\end{align*}
where $\delta_x$ denotes the configuration with unitary energy at site $x\in \mathbb T_N$ and zero elsewhere. 
{Here, note that we can easily extend the definition of the model with a long-range interaction rather than focusing on the nearest-neighbor case. This remark goes through with the other models that we will present hereinafter. }

\begin{Remark}[KMP model]
The choice $\mathfrak s=1/2$ leads to the classical model \cite{kipnis1982heat} where energy is redistributed according to a $Beta(1,1)$, that is, a continuous uniform distribution in $[0,1]$.
\end{Remark}

It is not hard to show that these models have a product invariant measure that we denote by $\nu_\rho^{\mathrm{gKMP}}$ which is also reversible under the dynamics, that is, the generator $L_N$ is self-adjoint with respect to a Gamma distribution with shape parameter $2\mathfrak s > 0$ and free scale parameter $\rho > 0$, where the common marginal of the measure $\nu_\rho$ satisfies
\begin{equation*}
\nu_{\rho}^{\mathrm{gKMP}}(\eta_x\le z) 
= \frac{1}{\Gamma(2\mathfrak s) \rho^{2\mathfrak s}} \int_0^z 
\eta^{2\mathfrak s-1} e^{-\eta / \rho } d\eta 
\end{equation*}
for any $x\in\mathbb T_N$ and $z\ge 0$. 
Note that, with respect to this parametrization, all the moments are given by 
\begin{equation}
\label{momentsKMP}
E_{\nu_\rho^{\mathrm{gKMP}}} \big[ (\eta_x^{\mathrm{gKMP}})^m \big] 
= \rho^m \dfrac{\Gamma(2\mathfrak s + m)}{\Gamma(2\mathfrak s)} 
\end{equation}
for $m\in\mathbb N_0$, where $E_{\nu_\rho^{\mathrm{gKMP}}}[\cdot]$ denotes the expectation with respect to the measure $\nu_\rho^{\mathrm{gKMP}}$ and we use the same notation for the other models.  

\begin{Remark}[Continuous Harmonic models] Another family of energy redistribution models has been recently studied in \cite{franceschini2023integrable}. These models represent the continuous counterparts of the Harmonic models described below in Subsection \ref{subsec:harm} and share a duality relationship with them. Since the Generalized KMP models and the continuous Harmonic models have identical transport coefficients — namely, constant diffusivity and convex mobility — it is believed that they exhibit similar macroscopic behavior. However, we defer the detailed study of the hydrodynamic limit for the continuous Harmonic models to future work, as their definition remains formal and the existence of the associated stochastic process still needs to be proved. 
\end{Remark}

\subsection{Particle models} 
\label{dsc model}
In this section, we introduce the two particle models of interest. Both can be considered special cases of mass migration processes \cite{fajfrova2016invariant} where multiple particles are allowed to jump together at the same time. The first one is the dual of the KMP model and the second one is a generalized zero-range process.

\subsubsection{Discrete KMP model}
The discrete KMP model appeared for the first time in \cite{kipnis1982heat} as the dual particle model of the KMP. Here particles between two nearest-neighbor sites are redistributed according to a discrete uniform distribution with rate $1$, namely the process has the  Markov generator
\begin{equation} 
\label{dkmpgen}
    L_N^{\mathrm{dKMP}} f(\eta) = \sum_{x \in \mathbb{T}_N} L_{x,x+1}^{\mathrm{dKMP}} f(\eta)
\end{equation}
where
\begin{equation*}
 L_{x,x+1}^{\mathrm{dKMP}} f(\eta) =  \frac{1}{ \eta_{x} + \eta_{x+1} +1 } \sum_{r=0}^{\eta_{x} + \eta_{x+1}}
\big[f(\eta_0, \ldots, \eta_{x-1}, r, \eta_x + \eta_{x+1} -r,\ldots \eta_{N-1}) - f(\eta)\big] 
\end{equation*}
for any $f:\Omega_N^{\mathrm{dKMP}}\to\mathbb R$, where $\Omega_N^{\mathrm{dKMP}}=\mathbb N_0^{\mathbb T_N}$.
As before, it will be convenient to split the action of the generator as $L_{x,x+1}^{\mathrm{dKMP}} = \nabla_{x,x+1}^{\mathrm{dKMP}} +\nabla_{x+1, x}^{\mathrm{dKMP}}$ where $\nabla_{x,x+1}^{\mathrm{dKMP}}$ and $\nabla_{x+1,x}^{\mathrm{dKMP}}$ are defined as follows:
\begin{equation*}
\nabla_{x,x+1}^{\mathrm{dKMP}} f (\eta) 
=  \frac{1}{ \eta_{x} + \eta_{x+1} +1 }  \sum_{k=1}^{\eta_x}  
\big[ f(\eta -k \delta_x + k \delta_{x+1}) - f(\eta)\big] 
\end{equation*}
and 
\begin{equation*}
\nabla_{x+1, x}^{\mathrm{dKMP}} f (\eta) 
 =  \dfrac{1}{  \eta_{x} + \eta_{x+1} +1 } \sum_{k=1}^{\eta_{x+1}} 
 \big[ f(\eta +k \delta_x - k \delta_{x+1}) - f(\eta)\big] \;,
\end{equation*}
where we denote by $\delta_x$ the configuration with just one particle in site $x \in \mathbb{T}_N$.
In other words, $k$ particles jump from site $x$ to site $y$, in general, at rate
\begin{equation}
\label{eq:dKMP_rates}
c^k_{x,y}(\eta)
= \dfrac{1}{ \eta_{x}+\eta_{y}+1 } 
\mathbf{1}_{\lbrace k \le \eta_x \rbrace} ,
\end{equation} 
noting the the long-range interaction can be defined analogously.

\begin{Remark}
In principle, one can define a family of models labeled by the spin $\mathfrak s>0$: these models were first obtained via a thermalization limit of the Symmetric Inclusion Process, see Section 5 of \cite{carinci2013duality}. In this case particles between two neighboring sites are redistributed according to a negative hypergeometric distribution and the jump rate of $k$ particles from site $x$ to site $y$ is given by
\begin{equation*}
c^k_{x,y}(\eta) 
= \dfrac{\Gamma(4\mathfrak s)}{\Gamma(2\mathfrak s)^2} \dfrac{\Gamma(2\mathfrak s+\eta_x-k) \Gamma(2\mathfrak s + \eta_{y} +k) \Gamma(\eta_x + \eta_{y} +1) }{  \Gamma(1+\eta_x-k) \Gamma(1+\eta_{y}+k) \Gamma(4\mathfrak s + \eta_{x} + \eta_{y}) } 
\mathbf{1}_{\{ k \le \eta_x\} }  .
\end{equation*}
Fixing $\mathfrak s=1/2$, the classical model in \eqref{dkmpgen} is recovered. 
\end{Remark}

It is not hard to show that the discrete KMP has a product invariant measure $\nu_\rho^{\mathrm{dKMP}}$ on $\Omega_N^{\mathrm{dKMP}}$ which is also reversible and its common marginal is given by the geometric distribution with mean $\rho>0$: 
\begin{equation*}
\nu_{\rho}^{\mathrm{dKMP}}(\eta_x=k) 
= \rho^{-1}
(1+1/\rho)^{-k-1} 
\end{equation*}
for each $x\in\mathbb T_N$ and $k\in \mathbb N_0$. 
Then, the generator $L_N$ is self-adjoint with respect to the measure $\nu_\rho^{\mathrm{dKMP}}$.
Note that, with respect to this parametrization, all the factorial moments are given by 
\begin{equation}
\label{momentsdKMP}
E_{\nu_\rho^{\mathrm{dKMP}}}\Big[
\prod_{k=1}^m(\eta_{x}-k+1)\Big]
= m! \rho^m 
\end{equation}
for each $m\in\mathbb N_0$.

\subsubsection{Harmonic model}\label{subsec:harm}
Finally, let us introduce the harmonic model on the particle-type configuration space $\Omega_N^{\mathrm{Harm}}=\mathbb N_0^{\mathbb T_N}$. 
Harmonic models are a family of interacting particle systems of generalized zero-range type labeled by the spin $\mathfrak s>0$. They were recently introduced in \cite{frassek2020non} using integrable non-compact quantum spin chains, see also \cite{capanna2024class} for the spin $\mathfrak s=1/2$ case. In this article, we restrict ourselves to the spin range $\mathfrak s \geq 1/2$ for technical reasons, since we need to use in the proof of our main theorem the fact that the process is attractive, which is proved in \cref{attract-dkmp} when $\mathfrak s \geq 1/2$.  
Now let us define the model by the following Markov generator given by 
\begin{equation}
\label{harmonicgen}
L_N^{\mathrm{Harm}} 
= \sum_{x \in \mathbb{T}_N} L_{x,x+1}^{\mathrm{Harm}} 
\end{equation}
where $L_{x,x+1}^{\mathrm{Harm}}=\nabla_{x,x+1}^{\mathrm{Harm}} + \nabla_{x+1,x}^{\mathrm{Harm}}$ with  
\begin{equation*}
\nabla_{x,x+1}^{\mathrm{Harm}} f(\eta) 
= \sum_{k=1}^{\eta_x} \dfrac{\Gamma(\eta_x +1) \Gamma(\eta_x - k + 2\mathfrak s)}{k \Gamma(\eta_x - k +1 ) \Gamma(\eta_x + 2\mathfrak s)}  
\big[ f(\eta -k\delta_x + k \delta_{x+1}) - f(\eta)\big] 
\end{equation*}
and
\begin{equation*}
\nabla_{x+1,x}^{\mathrm{Harm}} f(\eta) 
= \sum_{k=1}^{\eta_{x+1} }
\dfrac{\Gamma(\eta_{x+1} +1) \Gamma(\eta_{x+1} - k + 2\mathfrak s)}{k \Gamma(\eta_{x+1} - k +1 ) \Gamma(\eta_{x+1} + 2\mathfrak s)} \big[ f(\eta + k \delta_x - k \delta_{x+1} ) - f(\eta)\big] .
\end{equation*}
In other words, 
a number of $k$ particles jump from site $x$ to site $y$ with rate
\begin{equation}
\label{eq:harm_rates}
c^k_{x,y} (\eta) = \dfrac{\Gamma(\eta_x+1) \Gamma(\eta_x - k + 2\mathfrak s)}{k \Gamma(\eta_x-k+1) \Gamma(\eta_x + 2\mathfrak s)} \mathbf{1}_{\lbrace k \leq \eta_x \rbrace}
\end{equation}
where again noting that a long-range interaction can be defined analogously. 
Similarly to the other models, the Harmonic model exhibits stationarity with respect to a product measure with spatially homogeneous weight. 
In this case, for a given parameter $\rho>0$, the common marginal of the invariant measure, that we denote by $\nu_\rho^{\mathrm{Harm}}$, is given by a negative binomial distribution: 
\begin{equation*}
\nu_{\rho}^{\mathrm{Harm}}(\eta_x=k) 
=\Big( \frac{1}{1+\rho}\Big)^{2\mathfrak s}
\Big( \frac{\rho}{1+\rho}\Big)^{k} 
\frac{\Gamma(2\mathfrak s +k)}{k!\Gamma(2\mathfrak s)}   
\end{equation*}
for each $x\in\mathbb T_N$ and for each $k\in \mathbb N_0$. 
Then, the generator $L_N^{\mathrm{Harm}}$ is self-adjoint with respect to the measure $\nu_\rho^{\mathrm{Harm}}$. 
Note that, with respect to this parametrization, all the factorial moments are given by 
\begin{equation} 
\label{momentsHarmonic}
E_{\nu_{\rho}^{\mathrm{Harm}}} 
\Big[ \prod_{k=1}^m (\eta_x-k+1) \Big]
= \frac{\Gamma(2\mathfrak s+m)}{\Gamma(2\mathfrak s)} \rho^m
\end{equation}
for each $m\in\mathbb N_0$.

\begin{Remark}
    The Harmonic models arise as an interacting particle system from a mapping of the non-compact Heisenberg XXX chain in one dimension \cite{frassek2020non}. This integrable system was first solved in \cite{lipatov1993high} using the quantum inverse scattering method. 
  In the case of spin $\mathfrak s=1/2$, the jump rates of the corresponding stochastic process are significantly simplified, leading to a model that plays the same role of the symmetric exclusion process in the compact case.
It is remarkable that, unlike the family of exclusion models, where the only integrable model is the one that  allows up to one particle per site, i.e. $\mathfrak s=1/2$, here all the models of the family are integrable for any spin $\mathfrak s>0$. 
\end{Remark}

\subsection{Common properties}
\label{common}
In what follows, let $\eta^\sigma=\{ \eta^\sigma(t):t\ge 0\}$ be a Markov process generated by $L_N^\sigma$ on the configuration space $\Omega_N^\sigma$ for each $\sigma \in \{\mathrm{gKMP},\mathrm{dKMP},\mathrm{Harm}\}$. 
The three models introduced above exhibit several common features, and the proof of their hydrodynamic limits is established through universal computations, which makes the approach quite general.
First of all, it is easy to see that there is one conserved quantity which remains invariant under the dynamics: the total energy for KMP and the total number of particles for discrete KMP and Harmonic models: i.e. for all $t\ge0$,  
\begin{equation}
\label{conslaw}
\sum_{x \in \mathbb{T}_N} \eta^\sigma_x(t) 
= \sum_{x \in \mathbb{T}_N} \eta^\sigma_x(0)
\end{equation}
for each $\sigma\in \{ \mathrm{gKMP},\mathrm{dKMP},\mathrm{Harm}\}$. 
Note that the occupation variable $\eta_x^\sigma$, which is either a non-negative real number or a non-negative integer, is an unbounded quantity. 
Furthermore, these models are characterized by constant diffusivity and quadratic convex mobility. They can be expressed within an algebraic framework based on representations of the non-compact $\mathfrak{su}(1,1)$ Lie algebra and are consequently linked through Markov duality relations, see \cite{dualitybook} for the details. 
Additionally, all these models are of gradient type, namely the instantaneous current can be written as the discrete gradient of some local function.
Indeed, letting $W^\sigma_{x,x+1}$ be the instantaneous current from site $x$ to $x+1$ defined by $L_N^\sigma \eta^\sigma_x = W_{x-1,x}^\sigma-W_{x,x+1}^\sigma$, we have that  
\begin{equation}\label{pluto}
W^\sigma_{x,x+1}
=L^\sigma_{x,x+1} \eta_x^\sigma
= D_\sigma
(\eta_{x+1}^\sigma - \eta_x^\sigma)
\end{equation}
where $D_\sigma>0$ is the diffusion coefficient of the model. 
For our models, a direct computation shows that 
\begin{equation}
\label{eq:diffusion_coefficient_list}
D_{\mathrm{gKMP}}
= D_{\mathrm{dKMP}} = \frac{1}{2}, \quad
D_{\mathrm{Harm}} = \frac{1}{2\mathfrak s}.  
\end{equation}
As a consequence, we can write
\begin{equation}
\label{diff}
\begin{aligned}
N^2L_N^\sigma \eta_x^\sigma 
= N^2\big(L_{x,x+1}^\sigma \eta_x^\sigma + L_{x-1,x}^\sigma \eta_x^\sigma \big)
= D_\sigma \Delta_N \eta^\sigma_x
\end{aligned}
\end{equation}
where $\Delta_N$ denotes the discrete Laplacian:
\begin{equation*}
\Delta_N\eta^\sigma_x 
= N^2 (\eta^\sigma_{x+1} + \eta_{x-1}^\sigma- 2\eta_x^\sigma).
\end{equation*}
Indeed, for the generalized KMP, 
\begin{equation*}
\begin{aligned}
L^{\mathrm{gKMP}}_N\eta_x
= \int_0^1 \gamma_{\mathfrak s}(u)\big(u(\eta_x +\eta_{x+1}) -\eta_x \big) du
+ \int_0^1 \gamma_{\mathfrak s}(u)\big((1-u) (\eta_x +\eta_{x-1}) -\eta_x \big) du
\end{aligned}
\end{equation*}
from which \eqref{diff} follows, noting $\int_0^t \gamma_{\mathfrak s}(u) du = 1$ and $\int_0^t \gamma_{\mathfrak s}(u)u du=\Gamma(2\mathfrak s)^2/(2\Gamma(4\mathfrak s))$. 
Next, for the discrete KMP, we have 
\begin{equation*}
\begin{aligned}
L^{\mathrm{dKMP}}_N\eta_x
=\frac{1}{\eta_x+\eta_{x+1}+1} \sum_{r=0}^{\eta_x+\eta_{x+1}} (r-\eta_x) 
+ \frac{1}{\eta_x+\eta_{x-1}+1} \sum_{r=0}^{\eta_x+\eta_{x-1}} (\eta_{x-1}-r) ,
\end{aligned}
\end{equation*}
so that \eqref{diff} clearly holds.
Finally for the harmonic model, we compute 
\begin{equation*}
L^{\mathrm{Harm}}_N\eta_x
= \sum_{k=1}^{\eta_{x+1}}c^{\mathrm{Harm}}_x(\eta)k
+ \sum_{k=1}^{\eta_{x-1}}c^{\mathrm{Harm}}_x(\eta)k
-2 \sum_{k=1}^{\eta_{x}}c^{\mathrm{Harm}}_x(\eta)k
\end{equation*}
where 
\begin{equation*}
c_x^{\mathrm{Harm}}(\eta)
= \frac{\Gamma(\eta_x+1)\Gamma(\eta_x-k+2\mathfrak s)}{\Gamma(\eta_x+2\mathfrak s)\Gamma(\eta-k+1)}.
\end{equation*}
In this last case, the identity \eqref{diff} follows from an elementary identity for the Beta-binomial distribution (see \eqref{ident_1}). 
The computation \eqref{diff} gives rise to the viscosity term, in the diffusive time scaling, in the hydrodynamic equation.

\subsection{Statement of the main result: hydrodynamic limit}
Fix a time horizon $T > 0$ and consider a finite time window $[0,T]$. 
For each model $\sigma\in \{\mathrm{gKMP}, \mathrm{dKPM}, \mathrm{Harm}\}$, let $\mcb{D}([0,T],\Omega_N^\sigma)$ be the space of c\`{a}dl\`ag time trajectories with values in $\Omega_N^\sigma$ endowed with the Skorohod topology; and let $\mcb{M}_+$ be the space of non-negative Radon measures on $\mathbb{T}$ equipped with the weak topology. 
For each $\sigma$, define the empirical energy (or particle) measure $\pi^{N,\sigma}(\cdot)\in \mcb{M}_+$ by
\begin{equation*}
\pi^{N,\sigma}(\eta^\sigma,du) 
= \dfrac{1}{N} \sum_{x \in \mathbb{T}_N} \eta_x^\sigma \delta_{x/N} (du) 
\end{equation*}
where $\delta_{x/N}(\cdot)$ is the Dirac measure in $\mathbb{T}$ with its mass at $x/N$. 
We denote the integral of a test function $G : \mathbb{T} \to \mathbb{R}$ with respect to $\pi^{N,\sigma}$ by $\langle \pi^{N,\sigma} , G \rangle$:
\begin{equation*}
\langle \pi^{N,\sigma},G \rangle 
=  N^{-1} \sum_{x\in\mathbb T_N} \eta_x^\sigma 
G(x/N). 
\end{equation*}

Here let us recall the notion of association to a density profile from \cite[Definition 3.0.2]{kipnis1999scaling}.

\begin{defin}[Association to a profile]
A sequence of probability measures $\left( \mu_N \right)_{N \geq 1}$ on the configuration space $\Omega_N^\sigma$ is associated to a bounded measurable density profile $\rho_0:\mathbb{T} \to \mathbb{R}_+$ if for every continuous function $G:\mathbb T\to \mathbb R$ and $\delta>0$, it holds that
\begin{equation*}
\lim_{N \to \infty } \mu_N \bigg( \eta^\sigma \in \Omega_N^\sigma 
: \Big| \langle \pi^{N,\sigma}, G \rangle - \int_{\mathbb T} G(u) \rho_0(u) du \Big| > \delta  \bigg) = 0.
\end{equation*}
\end{defin}
Given a probability measure $\mu_N^\sigma$ on $\Omega_N^\sigma$, denote by $\mathbb{P}_{\mu_N^\sigma}$ the probability measure on $\mcb{D}([0,T],\Omega_N^\sigma)$ induced by the initial measure $\mu_N^\sigma$ and the Markov process $\eta_.^\sigma$; the corresponding expectation is denoted by {$\mathbb{E}_{\mu_N^{\sigma}}$}.
Moreover, let $\mcb{D}([0,T],\mcb{M}_+)$ be the space of c\`{a}dl\`{a}g trajectories endowed with the Skorohod topology and let $(\mathbb{Q}_N^\sigma)_{N \ge 1}$ be the sequence of probability measures on $\mcb{D}([0,T],\mcb{M}_+)$ induced by the initial measure $\mu^\sigma_N$ and the Markov process $\pi^{N,\sigma}_\cdot$. 
Here let us impose the following condition on the initial measure, in order to make use of attractiveness. 

\begin{assum}
\label{assum:initial_measure}
For each model $\sigma\in\{ \mathrm{gKMP},\mathrm{dKMP},\mathrm{Harm}\}$, assume that the initial measure $\mu_N^\sigma$ is associated to a bounded, measurable density profile $\rho_0$ and it is stochastically dominated by the invariant measure $\nu_{\hat\rho}^\sigma$ for some $\hat\rho>0$, i.e. $\mu_N^\sigma\leq \nu_{\hat\rho}^\sigma $.
\end{assum} Above  the partial order between two measures is defined naturally: $\mu_1\le \mu_2$ holds if $E_{\mu_1}[f] \le E_{\mu_2}[f]$ for any bounded monotone function $f$\footnote{A function $f:\Omega_N^\sigma\to\mathbb R$ is said to be monotone if for any $\eta,\xi\in\Omega_N^\sigma$ such that $\eta\leq \xi$ i.e. $\eta_x \le \xi_x$ for all $x \in \mathbb{T}_N$, $f(\eta)\leq f(\xi)$. }.

Now, to state the main result, let us recall the notion of weak solutions of the heat equation.

\begin{defin}
\label{def:weak_solu}
Let $\rho_0: \mathbb{T} \to \mathbb{R}_+ $ be a bounded measurable function. A measurable function $\rho: [0,T] \times \mathbb{T} \to \mathbb{R}_+$ is a weak solution to the heat equation with initial profile $\rho_0$
\begin{equation}  
\label{eq:HDL_heat_equation}
\begin{cases}
\begin{aligned}
& \partial_t \rho(t,u) = D\Delta \rho(t,u)
&& (t,u) \in [0,T] \times \mathbb{T},
\\
& \rho(0,u) = \rho_0(u) && u \in \mathbb{T}
\end{aligned}
\end{cases}
\end{equation}
if $\rho \in L^2 ([0,T]\times \mathbb T)$ and for all $t \in [0, T]$ and  $H \in C^{1,2}([0, T] \times \mathbb{T}) $ it holds
\begin{equation}\label{eq:int_sol}
\int_{\mathbb T} \rho(t,u)H(t,u) du
- \int_{\mathbb T} \rho_0(u)H(0,u) du 
- \int_0^t \int_{\mathbb T} \rho(s,u)(\partial_s + \Delta) H(s,u)du ds
= 0. 
\end{equation}
\end{defin}
Above the space $C^{1,2}([0,T]\times \mathbb T)$ denotes the set of real-valued functions defined on $[0,T]\times \mathbb T$  that are of class $C^1$ on the first variable and $C^2$ on the second variable. The interested reader can find a proof of existence and uniqueness of the weak solution of the heat equation in \cite[Theorem A.2.4.4]{kipnis1999scaling} for example.

Now, for all three models, the main result is given as follows. 

\begin{Theorem}[Hydrodynamic limit]
\label{HDL}
Let $\rho_0: \mathbb{T} \to \mathbb{R}_+ $ be a {bounded} measurable function and let $\{\mu_N\}_{N \geq 1}$ be a sequence of probability measures satisfying \cref{assum:initial_measure} with the profile $\rho_0$ and density $\hat\rho>0$.  
{Additionally, we assume that the spin satisfies $ \mathfrak s\ge 1/2$ for $\sigma=\mathrm{Harm}$.} 
Then, for each model $\sigma\in \{\mathrm{gKMP}, \mathrm{dKPM}, \mathrm{Harm}\}$, for any $t \in [0,T]$, $\delta>0$ and  every continuous function $G:\mathbb T\to \mathbb R$, it holds that
\begin{equation*}
\lim_{N \to \infty} 
\mathbb{P}_{\mu_N} \bigg( 
\eta^\sigma \in \mcb{D}([0,T],\Omega_N^\sigma)
: \Big| \frac{1}{N}\sum_{x \in \mathbb{T}_N} \eta_x^\sigma(N^2 t) 
G\Big(\frac{x}{N} \Big) - \int_{\mathbb T} G(u) \rho^\sigma(t,u)du
\Big| > \delta  \bigg) = 0 
\end{equation*}
where $\rho^\sigma: [0,T] \times \mathbb{T} \to \mathbb{R}_+$ is the unique weak solution of the heat equation \ref{eq:HDL_heat_equation} with $D= D_\sigma$ and the diffusion coefficient $D_\sigma$ is given for each model by \eqref{eq:diffusion_coefficient_list}. 
\end{Theorem}

\begin{Remark}
The lower bound for the spin $\mathfrak s\ge 1/2$ when $\sigma=\mathrm{Harm}$ is a technical condition, which is required to show the attractiveness for the harmonic model.
We are not sure whether attractiveness still holds when $\mathfrak s<1/2$, but the hydrodynamic limit should be proved by detouring attractiveness. 
\end{Remark}

\begin{Remark}
It is not hard to see that our results go through on a $d$-dimensional space $\mathbb T_N^d$ with generic $d\ge 1$. 
Moreover, we can also study the same problem on the full lattice $\mathbb Z^d$. 
In the present paper, we decided to work on the one-dimensional torus to avoid technical difficulties, but we believe that the results should be valid for these cases.
\end{Remark}

The next section is dedicated to showing another property that is shared by all our models, namely, the \emph{attractiveness}. This property will be useful in the proof of our main result since it allows us to estimate expectations with respect to a general measure $\mu_N^\sigma$ by expectations with respect to the invariant measure $\nu_{\rho}^\sigma$ with $\rho=\hat\rho$ appearing in Assumption \ref{assum:initial_measure}, provided one is considering expectations of a monotone increasing function of $\eta^\sigma$, as well as the stochastic domination by the measure $\nu^\sigma_{\hat\rho}$ at initial time.

\section{Attractiveness}
\label{attra}
Attractiveness is a powerful tool to study interacting systems, see \cite[Chapter 2]{liggett1985interacting}, \cite[Chapter 9]{kipnis1999scaling}.   
Here let us recall the definition (see \cite[Definition 2.5.1]{kipnis1999scaling}) of attractiveness. 

\begin{defin}
An interacting particle system with Markov semi-group $\{S(t):t \ge0\}$ is said to be attractive if it preserves the partial order of the system, that is, any probability measure $\mu_1,\mu_2$ on the configuration space satisfies $S(t)\mu_1\le S(t)\mu_2$ for any $t \ge0$ provided that $\mu_1\le \mu_2$. 
\end{defin}

In the following, we show that all models, generalized KMP, discrete KMP and harmonic models, are attractive. 
Let us begin from the attractiveness of the generalized KMP model. 
To this end we use the basic coupling argument which we describe below.

\begin{Theorem}[Attractivness for generalized KMP models] The generalized KMP models with generator given in \eqref{gkmpgen} are attractive. 
\end{Theorem}
\begin{proof}
Let $\eta (t)$ and $\xi(t)$ be two independent copies of the generalized KMP. We have to show that the process preserves the partial order through the coupled evolution in the following sense: for all $t \geq 0$, $\eta (t) \leq \xi(t)$ a.s. as soon as $\eta (0) \leq \xi (0)$. This means that we start with two independent copies of the process $\eta (t)$ and $\xi (t)$ such that for all $x \in \mathbb{T}_N$, it holds that $\eta_x (0) \leq \xi_x (0)$. The basic coupling consist in associating to each bond $(x,x+1)$ the same Poisson clock. When it rings for the first time, say at time $t_1$ we have the following updated occupation variables 
\begin{align*}
& \eta_x(t_1) = B \big[\eta_x(t_1^-) + \eta_{x+1}(t_1^-) \big], \quad 
\eta_{x+1}(t_1) = (1-B) \big[\eta_x(t_1^-) + \eta_{x+1}(t_1^-) \big], \\
& \xi_x(t_1) = B \big[\xi_x(t_1^-) + \xi_{x+1}(t_1^-) \big], \quad 
\xi_{x+1}(t_1) = (1-B) \big[\xi_x(t_1^-) + \xi_{x+1}(t_1^-)\big] 
\end{align*}
where $B \sim \mathrm{Beta}(2\mathfrak s,2\mathfrak s)$ is a Beta random variable with parameter $(2\mathfrak s,2\mathfrak s)$. Since $\eta_x(t_1^-) = \eta_x(0) \leq \xi_x(0) = \xi_x(t_1^-)$ and
$\eta_{x+1}(t_1^-) = \eta_{x+1}(0) \leq \xi_{x+1}(0) = \xi_{x+1}(t_1^-)$, we note that the partial order at time $t_1$ is trivially satisfied: $\eta_x(t_1) \leq  \xi_x(t_1)$ and $\eta_{x+1}(t_1) \le \xi_{x+1}(t_1)$ for all $x \in \mathbb{T}_N$. 
\end{proof}

Next, to show that our particle systems are attractive, we refer to \cite[Theorem 2.9]{gobron2010couplings}, where necessary and sufficient conditions on the transition rates that yield attractiveness are provided.

\begin{prop}[{Theorem 2.9 of \cite{gobron2010couplings}}]
\label{prop:attractiveness}
A process $\left( \eta_t \right)_{t \geq 0}$ is attractive if and only if, when the occupation variables are ordered, i.e. $\xi_x \le \zeta_x$ 
for all $x\in\mathbb T_N$, the following two conditions regarding the transition rates are satisfied

 \begin{itemize}
     \item for all $ \ell \geq 0$, 
     \begin{equation}\label{att1}
         \sum_{k' > \zeta_{x+1} - \xi_{x+1}+ \ell} c^{k'}_{x,x+1} (\xi) \leq    \sum_{\ell' > \ell} c^{\ell'}_{x,x+1} (\zeta);
     \end{equation}
     \item    for all $ k \geq 0$, 
     \begin{equation} \label{att2}
         \sum_{k' > k} c^{k'}_{x,x+1} (\xi) \geq    \sum_{\ell' > \zeta_x - \xi_x + k} c^{\ell'}_{x,x+1} (\zeta).
     \end{equation}
 \end{itemize}

\end{prop}

Applying \cref{prop:attractiveness}, we first show the attractiveness of the discrete KMP model. 
For sake of notation we adopt the following definitions for the occupation variables of the processes $\xi$ and $\zeta$ (as in \cite{gobron2010couplings}): define $\alpha,\beta,\gamma,\delta$ by 
\begin{equation}
\label{eq:variable_def_convention}
\alpha= \xi_x ,\quad
\beta=\xi_{x+1}, \quad
\gamma=\zeta_x,\quad
\delta=\zeta_{x+1}.
\end{equation}

\begin{Theorem}[Attractiveness for discrete KMP models]
The discrete KMP model with generator \eqref{dkmpgen} is attractive.
\end{Theorem}
\begin{proof}
For the discrete KMP model  {let us recall \eqref{eq:dKMP_rates}
and  check that equations \eqref{att1}
and \eqref{att2} are valid. }
The inequality \eqref{att1} reads, for all $\ell \geq 0$, 
\begin{equation*}
       \sum_{k' = \ell + 1 + \delta - \beta }^{\alpha} \dfrac{1}{\alpha + \beta +1}
    \leq 
      \sum_{\ell' = \ell+1}^{\gamma} \dfrac{1}{\gamma + \delta +1}. 
\end{equation*}
Then, computing the sums we have 
\begin{equation*}
\frac{\alpha + \beta - \ell - \delta}{\alpha + \beta +1}  
\le \frac{\gamma  - \ell}{\gamma + \delta +1}.
\end{equation*}
By rearranging the last display, we have that 
\begin{equation*}
(\alpha + \beta) (\delta +1) -\ell (\gamma + \delta)
\le (\gamma + \delta) (\delta +1) -\ell (\alpha + \beta),
\end{equation*}
which holds by the hypothesis $\alpha \leq \gamma $ and $\beta \leq \delta $.
For inequality \eqref{att2}, we need to prove that for all $k \geq 0$,
\begin{equation*}
 \sum_{k'= k+1}^{\alpha} \dfrac{1}{\alpha + \beta +1} \geq     \sum_{\ell'= \gamma - \alpha +k+1}^{\gamma} \dfrac{1}{\gamma + \delta +1} .
\end{equation*}
If $\a \le k$, both sides are $0$.
On the other hand, if $\a > k$, then computing the sums, one immediately sees that 
\begin{equation*}
\dfrac{\alpha -k}{\alpha + \beta +1} 
\ge \dfrac{\alpha -k}{\gamma + \delta +1} ,
\end{equation*}
which holds again by the hypothesis $\alpha \leq \gamma $ and $\beta \leq \delta $.
\end{proof}

\begin{Theorem}[Attractiveness for Harmonic models]
\label{attract-dkmp}
The Harmonic models with generator \eqref{harmonicgen} are attractive if $\mathfrak s\ge 1/2$.
\end{Theorem}
\begin{proof}
For the Harmonic models  let us recall \eqref{eq:harm_rates} and check that equations \eqref{att1} and \eqref{att2} hold.
Recall the notation \eqref{eq:variable_def_convention}. 
We start with showing \eqref{att1}, that is, when $\alpha \leq \gamma $ and $\beta \le \delta $ for all $\ell \geq 0$ it holds that
\begin{equation*}
    \sum_{k' = \ell + 1 + \delta - \beta }^{\alpha} \dfrac{1}{k'} \dfrac{\alpha! \Gamma(\alpha - k' + 2\mathfrak s)}{\Gamma(\alpha + 2\mathfrak s) (\alpha - k')!}
    \le
      \sum_{\ell' = \ell+1}^{\gamma} \dfrac{1}{\ell'} \dfrac{\gamma! \Gamma(\gamma - \ell' + 2\mathfrak s)}{\Gamma(\gamma + 2\mathfrak s) (\gamma - \ell')!}.
\end{equation*}
Since the support of the sum in the left-hand side in the last display is contained in the one in the right-hand side, the inequality is implied by
\begin{equation*}
\sum_{k' = \ell + 1 + \delta - \beta }^{\alpha} \dfrac{1}{k'} \dfrac{\Gamma(\alpha + 1) \Gamma(\alpha - k' + 2\mathfrak s)}{\Gamma(\alpha + 2\mathfrak s) \Gamma(\alpha - k' +1)}
\leq 
\sum_{k' = \ell + 1 + \delta - \beta }^{\alpha}
\dfrac{1}{k'} \dfrac{\Gamma(\gamma +1) \Gamma(\gamma - k' + 2\mathfrak s)}{\Gamma(\gamma + 2\mathfrak s) \Gamma(\gamma - k' +1)}
\end{equation*}
and the inequality holds if we show that the summand  
\begin{equation*}
    \varphi(n) =    \dfrac{1}{k'} \dfrac{\Gamma(n +1) \Gamma(n - k' + 2\mathfrak s)}{\Gamma(n + 2\mathfrak s) \Gamma(n - k' +1)}
\end{equation*}
is a non decreasing function in $n$ so that the hypothesis of $\alpha \leq \gamma$ implies $\varphi (\alpha)   \leq \varphi(\gamma)$ for $k' = \ell + 1 + \delta - \beta, \ldots, \alpha $.
We consider $\varphi$ as a function of a continuous variable $z$ and show that $\varphi'(z) \geq 0$. The sign of the derivative is determined by its numerator, namely 
we want to show that the quantity
\begin{align*}
& \big[ \Gamma'(z +1 ) \Gamma(z - k' +2\mathfrak s) + \Gamma(z+1) \Gamma'(z -k' +2\mathfrak s ) \big] 
\big[ \Gamma(z +2\mathfrak s) \Gamma(z -k' +1) \big]  \\
&-\big[ \Gamma(z +1) \Gamma(z -k' +2\mathfrak s)  \big] 
\big[ \Gamma'(z +2\mathfrak s ) \Gamma(z - k' +1) + \Gamma(z+2\mathfrak s) \Gamma'(z -k' +1 ) \big]  
\end{align*}
is non-negative for any $z \ge k'$.
We now divide the last quantity by the gamma function to reconstruct the digamma function, namely, the logarithm derivative of the gamma function:
$$\psi (z) =  \frac{d}{dz} \ln(\Gamma(z)) = \Gamma'(z)/\Gamma(z) . 
$$
Then, what we would like to show is the following: 
\begin{align*}
\dfrac{\Gamma'(z +1)}{\Gamma(z +1)} + \dfrac{\Gamma'(z -k' +2\mathfrak s)}{\Gamma(z -k' +2\mathfrak s)} 
\geq 
\dfrac{\Gamma'(z +2\mathfrak s)}{\Gamma(z +2\mathfrak s)} + \dfrac{\Gamma'(z -k' +1)}{\Gamma(z -k' +1)} .
\end{align*} 
Using the digamma function $\psi$, the last display reads
\begin{equation*}
\psi(z +1) -   \psi(z - k' +1) \geq   \psi(z +2\mathfrak s) -   \psi(z- k' +2\mathfrak s).
\end{equation*}
Again, if we show that $g(z)=\psi(z +1) - \psi(z - k' +1)$ is non increasing in $z$, we are done since $z+1 \le z + 2\mathfrak s$ for $2\mathfrak s \geq 1$. 
This involves the first polygamma function ($\psi^{(m)}(z) = \frac{d^m}{dz^m} \psi(z) = \frac{d^m}{dz^m} \ln(\Gamma(z)) $ with $m=1$),
which, for positive arguments, is a decreasing function, i.e.
$\psi'(z +1) \leq \psi'(z -k' +1) $. This
implies  $g'(z) = \psi'(z +1) -   \psi'(z - k' +1) \leq 0 $ and thus $g(z)$ non increasing for $z \ge k'$.

On the other hand, for \eqref{att2}, we have to show that for all $k \geq 0$
\begin{equation*}
    \sum_{k' = k + 1 }^{\alpha} \dfrac{1}{k'} \dfrac{\alpha! \Gamma(\alpha - k' + 2\mathfrak s)}{\Gamma(\alpha + 2\mathfrak s) (\alpha - k')!}
    \geq 
      \sum_{\ell' = k+1 + \gamma - \alpha}^{\gamma} \dfrac{1}{\ell'} \dfrac{\gamma! \Gamma(\gamma - \ell' + 2\mathfrak s)}{\Gamma(\gamma + 2\mathfrak s) (\gamma - \ell')!}.
\end{equation*}
Consider the following change of variable for the sum on the right-hand side of last display: $k' = \ell' - \gamma + \alpha$, then the previous inequality reads
\begin{equation*}
    \sum_{k' = k + 1 }^{\alpha} \dfrac{1}{k'} \dfrac{\alpha! \Gamma(\alpha - k' + 2\mathfrak s)}{\Gamma(\alpha + 2\mathfrak s) \Gamma(\alpha - k' +1)}
    \geq 
      \sum_{k' = k+1 }^{\alpha} \dfrac{1}{k' + \gamma - \alpha} \dfrac{\gamma! \Gamma(\alpha - k' + 2\mathfrak s)}{\Gamma(\gamma + 2\mathfrak s) \Gamma(\alpha - k' +1)}.
\end{equation*}
The above inequality holds  if we show that for all $k+1 \leq k' \le \alpha$,
\begin{equation*}
  \dfrac{1}{k'} \dfrac{\Gamma(\alpha+1) }{\Gamma(\alpha + 2\mathfrak s) }  
  \geq 
  \dfrac{1}{k' + \gamma - \alpha} \dfrac{\Gamma(\gamma +1) }{\Gamma(\gamma + 2\mathfrak s) }
\end{equation*}
namely, 
\begin{equation*}
    \Gamma(\alpha+1)k'\Gamma(\gamma + 2\mathfrak s) + \Gamma(\alpha +1)\Gamma(\gamma + 2\mathfrak s) (\gamma - \alpha) \geq \Gamma(\gamma +1) k' \Gamma (\alpha +2\mathfrak s)
\end{equation*}
which, since $\gamma - \alpha \geq 0$ holds if
\begin{equation*}
    \dfrac{\Gamma(\gamma + 2\mathfrak s)}{\Gamma(\gamma +1)} \geq  \dfrac{\Gamma(\alpha + 2\mathfrak s)}{\Gamma(\alpha +1)}, 
\end{equation*}
which is verified if the function
\begin{equation*}
f(n)= \dfrac{\Gamma(n + 2\mathfrak s)}{\Gamma(n +1)}
\end{equation*}
is non-decreasing in $n \geq 0$. 
As before, we study the sign of its derivative which is determined by the numerator:
\begin{equation*}
    \Gamma'(z+ 2\mathfrak s) \Gamma(z+1) - \Gamma(z +2\mathfrak s)\Gamma'(z+1).
\end{equation*}
Again notice that $f'(z) \geq 0$ since 
the digamma function is an increasing function for positive argument:
\begin{equation*}
  \psi(z+2\mathfrak s) = \dfrac{ \Gamma'(z+ 2\mathfrak s)}{\Gamma(z +2\mathfrak s)} \geq   \dfrac{ \Gamma'(z+ 1)}{\Gamma(z+1)} =  \psi(z+1) ,
\end{equation*}
for $2\mathfrak s \ge 1$. 
\end{proof}

\section{Proof of the hydrodynamic limit} 
\label{prooof}
In this section, we give a proof of \cref{HDL}.
The strategy of the proof is based on the so-called entropy method which was introduced in \cite{guo1988nonlinear}, see \cite{kipnis1999scaling} for the pedagogical description of the method. 
First, we prove that for each model $\sigma\in \{ \mathrm{gKMP},\mathrm{dKMP}, \mathrm{Harm}\}$, the sequence $(\mathbb{Q}^\sigma_N)_{N \geq 1}$ is tight with respect to the Skorohod topology in $\mcb{D}([0,T],\mcb{M}_+)$, see Section \ref{tightness}, so that the sequence has a limit point by taking a subsequence if necessary.  
Next, in Section \ref{clp} we characterize the limit points $\mathbb Q^*$ as the trajectory of measures that are absolutely continuous with respect to Lebesgue and whose density is the weak solution of the heat equation with the corresponding diffusion coefficient.
To this end, we firstly show that all the limit points are concentrated on trajectories of measures, that are continuous in time and are absolutely continuous with respect to the Lebesgue measure. 
Afterwards, we will show that the density of the limiting measure is a weak solution of the heat equation \eqref{eq:HDL_heat_equation}. 
The convergence of the full sequence follows from the uniqueness of the weak solution of the heat equation. 

\subsection{Tightness}
\label{tightness}
In order to show that the sequence of empirical measures $\{\pi^{N,\sigma}_\cdot\}_N$ is tight, 
it is enough to prove that for any $G \in C^2(\mathbb T)$
\begin{equation}\label{tight}
\lim_{A \to \infty} \limsup_{N \to \infty} 
\mathbb{P}_{\mu_N} \left(\sup_{0 \le t \le T}  \lvert \langle \pi_{t}^{N,\sigma}, G \rangle \rvert  > A \right) = 0
\end{equation}
and for all $\varepsilon > 0$ that 
\begin{equation}\label{tight2}
\lim_{\delta \to 0^+} \limsup_{N \to \infty} 
\mathbb{P}_{\mu_N} \left(\sup_{\substack{|t-s| \le \delta,\\ 0 \le t,s \le T}} 
\lvert \langle \pi_{t}^{N,\sigma}, G \rangle - \langle \pi_{s}^{N,\sigma}, G \rangle \rvert  > \varepsilon \right) = 0.
\end{equation}
(see \cite[Theorem 4.1.3, Remark 4.1.4]{kipnis1999scaling}.)
In view of Markov's inequality, the assertion \eqref{tight} boils down to the following:
\begin{equation*}
\sup_N \mathbb{E}_{\mu_N} \left[\sup_{0 \le t \le T}  \lvert \langle \pi_{t}^{N,\sigma}, G \rangle \rvert \right] < \infty.
\end{equation*}
By the conservation law, we have
\begin{equation*}
\mathbb{E}_{\mu_N} \left[\sup_{0 \le t \le T}  \lvert \langle \pi_{t}^{N,\sigma}, G \rangle \rvert \right] \le \|G\|_{\infty}\mathbb{E}_{\mu_N} \left[\sup_{0 \le t \le T}  \frac{1}{N} \sum_{x \in \mathbb{T}_N}\eta^{\sigma}_x(t) \right] = \|G\|_{\infty}\mathbb{E}_{\mu_N} \left[\frac{1}{N} \sum_{x \in \mathbb{T}_N}\eta^{\sigma}_x(0) \right].
\end{equation*}
The sum $\sum_{x \in \mathbb{T}_N}\eta^{\sigma}_x$ is a monotone function in $\eta^\sigma$, so we can bound the last term with the expectation with respect to $\nu_{\hat\rho}$ where $\hat\rho>0$ is the constant appearing in \cref{assum:initial_measure}:
\begin{equation}\label{unifbdd}
\mathbb{E}_{\mu_N}\Big[ \frac{1}{N}\sum_{x \in \mathbb T_N} \eta_x^\sigma (0) \Big] 
\le {E_{\nu_{\hat\rho}^\sigma}}[\eta_0^\sigma]{<+\infty},
\end{equation} 
hence we conclude \eqref{tight}.
To show \eqref{tight2}, recall from Dynkin's martingale formula, for any $t\ge 0$ and for any test function $G\in C^2(\mathbb T)$ that 
\begin{equation}\label{eq:Dynkin}
M_t^{N,\sigma}(G) 
= \langle \pi_t^{N,\sigma}, G \rangle 
- \langle \pi_0^{N,\sigma}, G \rangle 
- \int_0^t N^2 L_N^\sigma \langle \pi_s^{N,\sigma}, G \rangle ds
\end{equation}
and 
\begin{equation*}
N_t^{N,\sigma}(G) = M_t^{N,\sigma}(G)^2 - \langle M^{N,\sigma}(G) \rangle_{t} 
\end{equation*}
are mean-zero martingales with respect to the natural filtration of the process, where the quadratic variation is given by
\begin{equation}
\label{eq:qv}
\langle M^{N,\sigma}(G) \rangle_{t} 
=  \int_0^{t} \Upsilon_{\mathfrak s}^{N,\sigma}(G)ds  
\end{equation}
with 
$$
\Upsilon_{\mathfrak s}^{N,\sigma}(G)
= N^2 L_N \langle \pi_s^{N,\sigma}, G \rangle^2 
-2 N^2 \langle \pi_s^{N,\sigma}, G \rangle  L_N^\sigma \langle \pi_s^{N,\sigma}, G \rangle .
$$

Then, we can see that the carr\'e du champ $\Upsilon_{\mathfrak s}^{N,\sigma}(G)$ has the following bound commonly to all the models.

\begin{lemma}
\label{lem:qv_bound}
For each model $\sigma\in \{ \mathrm{gKMP},\mathrm{dKMP},\mathrm{Harm}\}$, for any $t\ge 0$ and for any test function $G \in C^2 (\mathbb{T})$, there exists some constant $C=C(G)>0$ such that
\begin{equation}
\label{eq:carre_du_champs_bound}
\Upsilon^{N,\sigma}_s(G)
\le \frac{C}{N^2}\sum_{x\in\mathbb T_N} \eta_x^\sigma(t)^2 .
\end{equation}
\end{lemma}

The proof of \cref{lem:qv_bound} is postponed to \cref{sec:universal_computations}. 
Now, in view of Markov's and Chebychev’s inequalities, the assertion \eqref{tight2} is reduced to the following.
\begin{prop}
\label{prop:tightness_key_estimates}
For each model $\sigma\in \{ \mathrm{gKMP},\mathrm{dKMP},\mathrm{Harm}\}$, for any test function $G \in C^2 (\mathbb{T})$, it holds that 
\begin{equation*}
\lim_{\delta \to 0^+} \limsup_{N \to \infty} 
\mathbb{E}_{\mu_N} \bigg[ \sup_{\substack{|t-s| \le \delta,\\ 0 \le t,s \le T}} \bigg| \int_{s}^{t} N^2 L^\sigma_N \langle \pi_r^{N,\sigma}, G \rangle dr \bigg| \bigg] = 0
\end{equation*}
and
\begin{equation*}
\lim_{\delta \to 0^+} \limsup_{N \to \infty} 
\mathbb{E}_{\mu_N} \Big[ \sup_{\substack{|t-s| \le \delta,\\ 0 \le t,s \le T}} \big( M_{t}^{N,\sigma}(G) - M_{s}^{N,\sigma}(G) \big)^2  \Big] = 0.
\end{equation*}
\end{prop}
\begin{proof}
For the first item, using the computation \eqref{diff}, 
\begin{equation*}
\begin{aligned}
 \mathbb{E}_{\mu_N} \bigg[  \bigg| 
\int_{s}^{t} N^2 L_N^\sigma \langle \pi_r^{N,\sigma}, G \rangle dr \bigg| \bigg] 
&=D_\sigma \mathbb{E}_{\mu_N} \bigg[ \bigg| \int_{s}^{t} \frac{1}{N} \sum_{x\in \mathbb{T}_N} \Delta_N G\Big(\frac{x}{N}\Big) \eta_x^\sigma(r) dr  \bigg| \bigg] \\ 
& \le D_\sigma|t-s| \| \Delta G\|_{L^\infty(\mathbb T)}
\mathbb{E}_{\mu_N}\Big[\frac{1}{N} \sum_{x \in \mathbb{T}_N} \eta^\sigma_x (0) \Big], 
\end{aligned}
\end{equation*}
where in the last inequality we used the conservation law \eqref{conslaw}. 
Above, $\Delta_N$ denotes the discrete one-dimensional Laplacian so that 
\begin{equation*}
\Delta_NG(x/N)
= N^2\big(G((x+1)/N) 
+ G((x-1)/N)
-2G(x/N) ) \big) . 
\end{equation*} 
Hence, by \eqref{unifbdd}, we conclude the proof.

For the second item, note that 
\begin{align*}
\mathbb{E}_{\mu_N} \bigg[ \sup_{\substack{|t-s| \le \delta,\\ 0  \le t,s \le T}} \big( M_{t}^{N,\sigma}(G) - M_{s}^{N,\sigma}(G) \big)^2  \bigg] &
\le \mathbb{E}_{\mu_N} \Big[ \sup_{\substack{|t-s| \le \delta,\\ 0 \le t,s \le T}} \big(2  M_{t}^{N,\sigma}(G)^2  +2  M_{s}^{N,\sigma}(G) ^2 \big)  \Big]  \\
 & \le 4 \mathbb{E}_{\mu_N} \Big[ \sup_{0 \le t \le T} M_{t}^{N,\sigma}(G)^2 \Big] \le 16 \mathbb{E}_{\mu_N} \Big[  M_{T}^{N,\sigma}(G)^2 \Big]
\end{align*}
where we used Doob's inequality in the last inequality. 
Since the martingale $N_t^{N,\sigma}(G)$ is mean-zero, we have 
\begin{equation} 
\label{pippo}
\mathbb{E}_{\mu_N} 
\Big[ M_{T}^{N,\sigma} (G)^2 
 \big] 
=  \mathbb{E}_{\mu_N} \Big[ 
\int_{0}^{T} \Upsilon_r^{N,\sigma}(G)dr \Big].
\end{equation}
Using the common bound \eqref{eq:carre_du_champs_bound}, we have that 
\begin{equation*}
\begin{aligned}
\mathbb{E}_{\mu_N} \Big[ \int_{0}^{T} \Upsilon_r^{N,\sigma}(G)ds \Big] 
& \le C \mathbb{E}_{\mu_N}  \Big[ \int_{0}^{T} \frac{1}{N^2} \sum_{x \in \mathbb{T}_N} \eta_ x^\sigma(r)^2 dr \Big]. 
\end{aligned}
\end{equation*}
Then, since the function $\sum_{x \in \mathbb{T}_N} ( \eta_ x^\sigma)^2$ is a monotone function, by the attractiveness and stochastic domination, we have
\begin{equation*}
\mathbb{E}_{\mu_N} \bigg[ \int_{0}^{T} \Upsilon_r^{N,\sigma}(G)ds \bigg] 
\le C {\mathbb{E}_{\nu_{\hat\rho}^\sigma}} \Big[ \int_{0}^{T} \frac{1}{N^2} \sum_{x \in \mathbb{T}_N} \eta_ x^\sigma(r)^2 dr \Big]
= CT\frac{1}{N} {E_{\nu_{\hat\rho}^\sigma}} \big[  (\eta_0^\sigma)^2 \big], 
\end{equation*}
where $\hat{\rho}$ is the parameter in Assumption \ref{assum:initial_measure}. Hence, we are done.
\end{proof}

\subsection{Characterization of  limit points}\label{clp}
In this part, we characterize the limit points of the sequence of empirical measures $\{\pi^{N,\sigma}_\cdot\}_N$, or $\{\mathbb Q^\sigma_N\}_N$ in terms of its distribution on the path space.  
Throughout this section, let us omit the dependency on $\sigma$, since the proof goes through exactly in the same way for all three models.  

\subsubsection{The limit point of the sequence has continuous time trajectories}
In order to show that any limiting point of the sequence $\lbrace \pi^N_{.} \rbrace_{N \geq 1}$ is  continuous with respect to the time variable we will show the next result. 
\begin{lemma}
For any $G\in C(\mathbb T)$, we have that
\begin{equation*}
\lim_{N \to \infty} \mathbb{E}_{\mu_N} \Big[ 
\sup_{t\ge0} \big| \langle \pi^N_t,G \rangle - \langle \pi^N_{t^-},G \rangle \big| \Big] = 0. 
\end{equation*}
\end{lemma}
\begin{proof}
First, we show the assertion, assuming that a test function $G$ is in the class $C^1(\mathbb T)$. 
To show the assertion, suppose that there is  an exchange of energy or a jump at time $t$ at some bond $(y,y+1)$.
Then 
\begin{equation*}
\begin{aligned}
\langle \pi^N_t,G \rangle - \langle \pi^N_{t^-},G \rangle  
= \frac{1}{N} \sum_{x \in \mathbb T_N} 
G \Big(\frac{x}{N} \Big) 
\big[ \eta_{x} (t) - \eta_{x} (t^-)\big] 
= \frac{1}{N^2} \nabla^+_N G\Big(\frac{y}{N}\Big)\big[ \eta_{y} (t^-) - \eta_y (t) \big]
\end{aligned} 
\end{equation*}
since the conserved dynamics yields $\eta_y(t^-)-\eta_y(t) = \eta_{y+1}(t) - \eta_{y+1}(t^-)$ after an exchange of energy or jump of particles. 
Above, we introduced the discrete gradient $\nabla^+_NG(y/N)=N[G((y+1)/N))-G(y/N)]$. 
Note that the last display is absolutely bounded as follows: 
\begin{equation*}
\begin{aligned}
\big|\langle \pi^N_t,G \rangle - \langle \pi^N_{t^-},G \rangle \big|
& \le \frac{1}{N^2} \Big(\max_{x\in\mathbb T_N}|\nabla_N^+G(x/N)| \Big)
\big|\eta_{y} (t^-) - \eta_y (t) \big| \\
& \le 2\| \nabla G\|_{L^\infty(\mathbb T)} \dfrac{1}{N^2} \big( \eta_{y} (t^-) + \eta_y (t)  \big).
\end{aligned}
\end{equation*}
Now, if we consider the whole sum, we can use the conservation law \eqref{conslaw}:
\begin{equation*}
\lvert \langle \pi^N_t,G \rangle - \langle \pi^N_{t^-},G \rangle \rvert \leq \frac{2\|\nabla G\|_{L^\infty(\mathbb T)}}{N^2} \sum_{y \in \mathbb{T}_N}\left( \eta_{y} (t^-) + \eta_y (t)  \right) 
\le \frac{4 \| \nabla G\|_{L^\infty(\mathbb T)}}{N^2} \sum_{y \in \mathbb{T}_N} \eta_{y}(0). 
\end{equation*} 
Note that the supremum in time at this point does not play any role and we concluded the proof thanks to Assumption \ref{assum:initial_measure}:
\begin{equation*}
\mathbb{E}_{\mu_N} \Big[ \sup_{t \ge 0 } \ 
\lvert \langle \pi^N_t,G \rangle - \langle \pi^N_{t^-},G \rangle \rvert \Big]
\le \frac{4\|\nabla G\|_{L^\infty(\mathbb T)}}{N} E_{\nu_{\hat{\rho}}} \Big[ \dfrac{1}{N}\sum_{y \in \mathbb{T}_N} \eta_{y} (0)  \Big], 
\end{equation*}
which converges to zero since $E_{\nu_{\hat{\rho}}}[\eta_y] $ is of order one for all $y \in \mathbb{T}_N$. 
Hence, the desired assertion is proved for any $G\in C^1(\mathbb T)$. 
To show that we can extend the assertion to any continuous function $G\in C(\mathbb T)$, take any $G_\varepsilon \in C^1(\mathbb T)$ such that $\| G_\varepsilon- G\|_{L^\infty(\mathbb T)} \le \varepsilon$ for any $\varepsilon>0$.   
Then, we have the bound 
\begin{equation*}
\big| \langle \pi^N_t,G\rangle - \langle \pi^N_t,G_\varepsilon \rangle \big|
\le \| G - G_\varepsilon \|_{L^\infty(\mathbb T)} 
\frac{1}{N}\sum_{x\in\mathbb T_N} \eta_x(t)
\end{equation*}
where we used the positivity of the occupation variable.
According to the conservation law and taking the expectation of last display, we conclude that we can bound that expectation  by $C\varepsilon$ for some constant $C>0$. 
Therefore, using the previous bound for $C^1(\mathbb T)$-test function, we have that 
\begin{equation*}
\begin{aligned}
\limsup_{N \to \infty}\mathbb{E}_{\mu_N} \Big[ 
\sup_{t\ge0} \big| \langle \pi^N_t,G \rangle - \langle \pi^N_{t^-},G \rangle \big| \Big] 
\le C\varepsilon,
\end{aligned}
\end{equation*}
from which we conclude the proof, since $\varepsilon>0$ is arbitrary. 
\end{proof}

\subsubsection{Absolute continuity of the limit measure}
In what follows, we will show that all limit points $\mathbb Q^*$ of the sequence of empirical measures are absolutely continuous with respect to the Lebesgue measure for all $t \in [0,T]$ with probability $1$. 
Here, let us begin by showing that the empirical measure $\pi_t$ has a density at any fixed time $t$ with probability one. 
To that end, from Proposition \ref{prop:abs_cont}, it is enough to show the following bound $\mathbb Q^*$-almost surely:
\begin{equation}
\label{eq:abs_conti_estimate}
\mathbb{Q}^* \Big( |\langle \pi_t,G \rangle| \le \hat\rho\|G\|_{L^1(\mathbb T)} \Big) 
=1  
\end{equation}
for any $t\in[0,T]$ and for any continuous function $G\in C(\mathbb T)$. Therefore, take a continuous function $G$ on $\mathbb{T}$. Since $\langle \pi^N_t,|G|\rangle$ is a monotone function, for any $\varepsilon >0$,
\begin{equation*}
\mathbb{P}_{\mu_N} \Big(   
\langle \pi^N_t,|G| \rangle - \hat\rho\| G\|_{L^1(\mathbb T)} >  \varepsilon \Big) 
\le \mathbb{P}_{\nu_{\hat\rho}} \Big(  
\langle \pi^N_t,|G| \rangle - \hat\rho \|G\|_{L^1(\mathbb T)} > \varepsilon \Big). 
\end{equation*}
Thus, by the law of large numbers, we have that 
\begin{equation*}
\limsup_{N\to\infty}
\mathbb{P}_{\nu_{\hat\rho}} \Big(   
\langle \pi^N_t,|G| \rangle- \hat\rho\|G\|_{L^1(\mathbb T)} > \varepsilon \Big) = 0. 
\end{equation*}
We observe that above we proved tightness of  the sequence $\{\mathbb Q_N\}_N$ with respect to the Skorohod topology of $\mathcal D([0,T], \mcb M_+)$.  In the previous subsection we proved that the limit point is supported on trajectories of measures that are continuous in time. Therefore, tightness also holds with respect to the uniform topology see \cite[Section 12, page 124]{billingsley2013convergence}. As a consequence,   since the set $\{ \pi^N_\cdot :  \ \langle \pi^N_t,|G| \rangle - \hat\rho\|G\|_{L^1(\mathbb T)} > \varepsilon\}$ is an open set with respect to the uniform topology,  we have
\begin{equation*}
\mathbb{Q}^* \Big(   
\langle \pi_t,|G| \rangle - \hat\rho \|G\|_{L^1(\mathbb T)} > \varepsilon \Big) 
\le \liminf_{N\to \infty} \mathbb{Q}_N \Big(  
\langle\pi_t,|G|\rangle - \hat\rho\|G\|_{L^1(\mathbb T)}  >  \varepsilon \Big) =0.
\end{equation*}
Hence, the desired assertion \eqref{eq:abs_conti_estimate} holds for any fixed time $t\in [0,T]$. 

Finally, to elevate the assertion up to now to measure-valued process, {first} note that 
\begin{equation*}
\mathbb Q^*
\Big( \sup_{t\in \mathbb Q{\cap[0,T]}} |\langle \pi_t, G \rangle| \le \hat\rho\|G\|_{L^1(\mathbb T)} 
\Big) = 1. 
\end{equation*}

Then, since the limit point $\mathbb Q^*$ is concentrated on continuous trajectories, we have that 
\begin{equation*}
\mathbb{Q}^* \Big( |\langle \pi_t,G \rangle| \le \hat\rho\|G\|_{L^1(\mathbb 
T)} , \ 0 \le t \le T\Big)=1.
\end{equation*}
Therefore, all limit points are concentrated on absolutely continuous trajectories with respect to the Lebesgue measure:
\begin{equation*}
\mathbb{Q}^*\Big( \pi ; \pi_t(du)=\rho(t,u)du, \ 0 \le t \le T \Big) =1.
\end{equation*}

\subsubsection{Density as a solution to the heat equation}
Now we finally prove that the density is a solution to the heat equation in the sense given in Definition \ref{def:weak_solu}. We give a short proof and more details can be seen e.g. in \cite{kipnis1999scaling}.
Recall \eqref{eq:Dynkin} and \eqref{diff}. Consider $G\in{C^{1,2}([0,T]\times \mathbb T})$. Therefore, from a summation by parts, we can rewrite the Dynkin martingale \eqref{eq:Dynkin} as
\begin{equation*}
M_t^{N}(G) 
= \langle \pi_t^{N}, G_t \rangle 
- \langle \pi_0^{N}, G_0\rangle 
- \int_0^t  \langle \pi_s^{N}, D \Delta_N G_s+\partial_s G_s \rangle ds.
\end{equation*}

Note that we have the bound 
\begin{equation*}
\begin{aligned}
\mathbb P_{\mu_N} \Big(
\sup_{0\le t\le T}
| M^N_t(G)| > \varepsilon \Big)
\le 4 \varepsilon^{-2} \mathbb E_{\mu_N}[M^N_T(G)^2]
= 4 \varepsilon^{-2}
\mathbb E_{\mu_N}\bigg[
\int_0^T \Upsilon^N_t(G)dt \bigg]
\end{aligned}
\end{equation*}
where we used Chebychev's and Doob's inequality in the first estimate. At this point it is enough to recall the proof of Proposition \ref{prop:tightness_key_estimates} to conclude that last term clearly vanishes as $N\to\infty$. We observe that from this result the martingale term does not contribute to the limit and thus the limiting equation is deterministic. 
Therefore, all limit points $\mathbb Q^*$ are concentrated on trajectories satisfying
\begin{equation*}
\langle \pi_t, G_t\rangle
= \langle \pi_0, G_0\rangle
+ \int_0^t \langle \pi_s, D{\Delta}G_s+\partial_s G_s\rangle ds 
\end{equation*}
for each $t\in [0,T]$ and for any $G\in C^{1,2}([0,T]\times\mathbb T)$ where we set $G_t(\cdot)=G(t,\cdot)$. 
Moreover, since the limit measure satisfies for all $t$:  $\pi_t(du)=\rho(t,u)du$ then the last display becomes exactly equal to \eqref{eq:int_sol}.

\subsection{Completion of the proof}
Now, we are in a position to complete the proof of \cref{HDL}. 
First, it is not difficult to show that all limit points of the sequence $\{\mathbb Q_N\}_{N\in\mathbb N}$ are concentrated on trajectories $\rho_0(u)du$ at time $0$. 
Thus, the previous argument assures that all limit points are concentrated on trajectories of absolutely continuous  measures with respect to the Lebesgue measure, i.e.  $\pi_t(du) = \rho(t,u) du$, where the density $\rho(t,u)$ is a weak solution of \eqref{eq:HDL_heat_equation}.
Moreover, from the uniqueness of the weak solution of the hydrodynamic equation, the convergence takes place along the full sequence, without taking any subsequence. 
Finally, since the limiting process $\{\pi_t(du): t \ge0\}$ is continuous in time, then  the projection $t\mapsto \langle \pi_t,G\rangle$ is continuous for any $G\in C(\mathbb T)$. 
Hence, by the continuous mapping theorem,  the sequence of probability measures $\{\pi^N_t\}_{N\in\mathbb N}$ converges in distribution to a deterministic measure $\rho(t,u)du$ for each fixed time $t$. 
Since convergence in distribution to a deterministic variable yields the convergence in probability, we complete the proof.

\appendix
\section{Proof of Lemma \ref{lem:qv_bound}} 
\label{sec:universal_computations}
In order to prove \cref{lem:qv_bound}, recall that we needed to compute the quadratic variation \eqref{eq:qv} of the Dynkin’s martingale. 
Here, for all models $\sigma\in\{\mathrm{gKMP},\mathrm{dKMP},\mathrm{Harm}\}$, let us give a universal feature that emerges in the computation of the carr\'e du champ $\Upsilon^{N,\sigma}_t(G)$ and deduces the bound \eqref{eq:carre_du_champs_bound}, which is used in the proof of tightness.
As a byproduct, we will give a proof of \cref{lem:qv_bound}. 
To this end, since all our dynamics are of nearest-neighbor type, we firstly note that 
\begin{equation}
\label{integrand}
\begin{aligned}
\Upsilon^{N,\sigma}_t(G)
&= \sum_{x,y \in \mathbb{T}_N} 
G\Big( \frac{x}{N}\Big) G\Big( \frac{y}{N}\Big) 
\big[ L_N^\sigma \big(\eta_x^\sigma \eta_y^\sigma\big)- 2 \eta_x^\sigma L_N^\sigma \eta_y^\sigma \big] \\ 
& =  \sum_{x \in \mathbb{T}_N} 
G\Big( \frac{x}{N}\Big)^2
\big( L_N^\sigma (\eta_x^\sigma)^2 - 2 \eta_x^\sigma L_N^\sigma \eta_x^\sigma \big) \\ 
&\quad+ 2\sum_{x \in \mathbb T_N} 
G\Big( \frac{x}{N}\Big) G\Big( \frac{x+1}{N}\Big)
\big( L_N^\sigma (\eta_x^\sigma\eta_{x+1}^\sigma) 
- \eta_x^\sigma L_N^\sigma \eta_{x+1}^\sigma 
- \eta_{x+1}^\sigma L_N^\sigma \eta_x^\sigma \big) 
\end{aligned}
\end{equation}
here in what follows, we omit the dependency on time in the occupation variable as far as there is not any confusion. 
The terms $\L_N^\sigma \eta_x^\sigma$ are considered in equation \eqref{diff}, so we are left to evaluate $L_N^\sigma (\eta_x^\sigma)^2 $ and $L_{N}^\sigma (\eta_x^\sigma \eta_{x+1}^\sigma)$ .
For the first one we have
\begin{align*}
L_N^\sigma (\eta_x^\sigma)^2 
= L_{x,x+1}^\sigma (\eta_x^\sigma)^2 
+ L_{x-1,x}^\sigma (\eta_x^\sigma)^2. 
\end{align*}
Then, we compute 
\begin{align*}
L_{x,x+1}^\sigma \eta_x^2 
&= L_{x,x+1}^\sigma \big( (\eta_x^\sigma)^2 + \eta_x^\sigma \eta_{x+1}^\sigma \big) 
- L_{x,x+1}^\sigma ( \eta_x^\sigma \eta_{x+1}^\sigma ) \\ 
& = ( \eta_x^\sigma + \eta_{x+1}^\sigma) L_{x,x+1}^\sigma \eta_x^\sigma 
- L_{x,x+1}^\sigma (\eta_x^\sigma \eta_{x+1}^\sigma ) \\ 
& = 
D_\sigma \big( (\eta_{x+1}^\sigma)^2 - (\eta_x^\sigma)^2 \big) 
-  L_{x,x+1}^\sigma 
(\eta_x^\sigma \eta_{x+1}^\sigma) .
\end{align*}
where we used {equation \eqref{pluto}} and the fact that 
$$
L_{x,x+1}^\sigma 
\big( F(\eta^\sigma)G(\eta_x^\sigma,\eta_{x+1}^\sigma)\big) 
= F(\eta^\sigma) L_{x,x+1}^\sigma G(\eta_x^\sigma,\eta_{x+1}^\sigma)
$$ 
for any $F(\eta^\sigma)$ which is a function of $\eta_y^\sigma$ with $y \neq x, x+1$ and $\eta_x^\sigma+\eta_{x+1}^\sigma$. 
Similarly, we have that 
\begin{align*}
L_{x-1,x}^\sigma (\eta_x^\sigma)^2 
= D_\sigma \big( (\eta_{x-1}^\sigma)^2 - (\eta_x^\sigma)^2 \big) 
-  L_{x-1,x}^\sigma(\eta_{x-1}^\sigma \eta_{x}^\sigma).
\end{align*}
This, together with \eqref{diff}, leads to 
\begin{equation*}
\begin{aligned}
& L_N^\sigma (\eta_x^\sigma)^2 
-2\eta_x^\sigma L_N^\sigma \eta_x^\sigma \\
&\quad= D_\sigma\big( (\eta_{x+1}^\sigma)^2 - (\eta_x^\sigma)^2 \big) 
- L_{x,x+1}^\sigma (\eta_x^\sigma \eta_{x+1}^\sigma) 
-2D_\sigma \eta_x^\sigma (\eta_{x+1}^\sigma - \eta_x^\sigma) \\
&\qquad + D_\sigma \big( (\eta_{x-1}^\sigma)^2 - (\eta_x^\sigma)^2 \big) 
- L_{x-1,x}^\sigma \eta_x^\sigma \eta_{x-1}^\sigma 
-2D_\sigma \eta_x^\sigma (\eta_{x-1}^\sigma - \eta_x^\sigma) \\
&\quad=  D_\sigma (\eta_{x+1}^\sigma - \eta_x^\sigma)^2 
- L_{x,x+1}^\sigma (\eta_x^\sigma \eta_{x+1}^\sigma) 
+ D_\sigma (\eta_{x-1}^\sigma - \eta_x^\sigma)^2 
- L_{x-1,x}^\sigma (\eta_x^\sigma \eta_{x-1}^\sigma) . 
\end{aligned}
\end{equation*}
On the other hand, for the last line of \eqref{integrand} we have that
\begin{equation*}
\begin{aligned}
& L_N^\sigma (\eta_x^\sigma\eta_{x+1}^\sigma) 
-  \eta_x^\sigma L_N^\sigma \eta_{x+1}^\sigma 
- \eta_{x+1}^\sigma L_N^\sigma \eta_{x}^\sigma \\
&\quad= L_{x,x+1}^\sigma (\eta_x^\sigma \eta_{x+1}^\sigma) 
- \eta_x^\sigma L_{x,x+1}^\sigma \eta_{x+1}^\sigma 
- \eta_{x+1}^\sigma L_{x,x+1}^\sigma \eta_x^\sigma \\
&\quad = L_{x,x+1}^\sigma (\eta_x^\sigma \eta_{x+1}^\sigma) 
- \eta_x^\sigma D_\sigma (\eta_x^\sigma - \eta_{x+1}^\sigma) 
- \eta_{x+1}^\sigma D_\sigma (\eta_{x+1}^\sigma - \eta_{x}^\sigma) \\
&\quad= L_{x,x+1}^\sigma (\eta_x^\sigma \eta_{x+1}^\sigma) 
-D_\sigma (\eta_x^\sigma - \eta_{x+1}^\sigma)^2
\end{aligned}
\end{equation*}
where in the second identity we used the fact that 
$$ 
L_{x,x+1}^\sigma \eta_{x+1}^\sigma 
= L_{x,x+1}^\sigma ( \eta_x^\sigma + \eta_{x+1}^\sigma - \eta_x^\sigma ) 
= -  L_{x,x+1}^\sigma \eta_x^\sigma 
= D_\sigma (\eta_x^\sigma - \eta_{x+1}^\sigma) .
$$ 
This allows us to rewrite the integrand part of the quadratic variation as
\begin{align*}
\Upsilon_t^{N,\sigma}(G) 
= \frac{1}{N^2}\sum_{x \in \mathbb{T}_N} 
\big( \nabla^+_N G(x/N) \big)^2 
\big[ D_\sigma ( \eta_x^\sigma - \eta_{x+1}^\sigma)^2 
- L_{x,x+1}^\sigma (\eta_x^\sigma \eta_{x+1}^\sigma) \big]. 
\end{align*}
Now, we are left to compute the quantity $L^\sigma_{x,x+1}(\eta^\sigma_x \eta^\sigma_{x+1})$.   
Here, as we shall see below, we claim that for each model $\sigma$, we have the following bound:  
\begin{equation}
\label{eq:carre_du_champ_estimate_key_bound}
D_\sigma( \eta_x^\sigma - \eta_{x+1}^\sigma)^2 
- L_{x,x+1}^\sigma (\eta_x^\sigma\eta_{x+1}^\sigma) 
\le D_\sigma \big((\eta_x^\sigma)^2 + (\eta_{x+1}^\sigma)^2 \big) .
\end{equation}
As a consequence 
\begin{equation*}
\begin{aligned}
\Upsilon_t^{N,\sigma}(G)  
& = \frac{1}{N^2} \sum_{x \in \mathbb{T}_N} 
\big( \nabla^+_N G(x/N)\big)^2 
\big[D_\sigma( \eta_x^\sigma - \eta_{x+1}^\sigma)^2 
- L_{x,x+1}^\sigma(\eta_x^\sigma \eta_{x+1}^\sigma) \big]  \\ 
&\le \frac{D_\sigma}{N^2}\sum_{x \in \mathbb{T}_N} 
\big( \nabla^+_N G(x/N)\big)^2 
\big[ (\eta_x^\sigma)^2 + (\eta_{x+1}^\sigma)^2\big] 
\le \frac{C}{N^2} \sum_{x \in \mathbb{T}_N} (\eta_x^\sigma)^2
\end{aligned}
\end{equation*}
with some constant $C=C(G)>0$. It now remains to  prove the claim i.e. the key bound \eqref{eq:carre_du_champ_estimate_key_bound} for each model $\sigma\in\{ \mathrm{gKMP},\mathrm{dKMP},\mathrm{Harm}\}$. This ends the proof of the lemma. 
Below, let us give a proof of the estimate \eqref{eq:carre_du_champ_estimate_key_bound} separately for each model.

\subsection{Generalized KMP}
For the generalized KMP model $\eta=\eta^{\mathrm{gKMP}}$, note that 
\begin{equation*}
L_{x,x+1}^{\mathrm{gKMP}} (\eta_x \eta_{x+1})
= (\eta_x + \eta_{x+1})^2 I(\mathfrak s) - \eta_x \eta_{x+1} 
\end{equation*}
where, recalling the definition of $\gamma_{\mathfrak s}(u)$ given in \eqref{eq:dkmp_redistribution_weight} and $I(\mathfrak s)$ is defined by 
\begin{equation*}
I(\mathfrak s) = \int_0^1 \gamma_{\mathfrak s}(u)u(1-u)du
= \dfrac{\Gamma(2\mathfrak s+1) \Gamma(2\mathfrak s+1)}{\Gamma(4\mathfrak s +2)}
\frac{\Gamma(4\mathfrak s)}{\Gamma(2\mathfrak s)\Gamma(2\mathfrak s)}
= \frac{s}{4\mathfrak s+1} . 
\end{equation*}
Thus, recalling $D_{\mathrm{gKMP}}=1/2$, we get the bound 
\begin{equation*}
\begin{aligned}
D_{\mathrm{gKMP}} ( \eta_x - \eta_{x+1})^2 
- L_{x,x+1}^{\mathrm{gKMP}}(\eta_x \eta_{x+1})
& = (D_{\mathrm{gKMP}} - I(\mathfrak s)) (\eta_x^2 + \eta_{x+1}^2) \\
& \le D_{\mathrm{gKMP}} (\eta_x^2 + \eta_{x+1}^2) .  
\end{aligned}
\end{equation*}

\subsection{Discrete KMP}
Next, for discrete KMP $\eta=\eta^{\mathrm{dKMP}}$ on the configuration space $\Omega_N^{\mathrm{dKMP}}=\mathbb N_0^{\mathbb T_N}$, we get that 
\begin{equation*}
\begin{aligned}
L_{x,x+1}^{\mathrm{dKMP}} (\eta_x \eta_{x+1})
&= \frac{1}{\eta_x + \eta_{x+1} +1}
\sum_{r=0}^{\eta_x + \eta_{x+1}} r (\eta_x + \eta_{x+1} -r) - \eta_x \eta_{x+1} \\
&= \frac{\eta_x^2}{6} + \frac{\eta_{x+1}^2}{6} - \dfrac{2}{3} \eta_x \eta_{x+1} - \dfrac{1}{6} \eta_x - \dfrac{1}{6}  \eta_{x+1}. 
\end{aligned}
\end{equation*}
This means, recalling from \eqref{eq:diffusion_coefficient_list} that $D_{\mathrm{dKMP}}=1/2$, we have the desired bound 
\begin{equation*}
D_{\mathrm{dKMP}} (\eta_x - \eta_{x+1})^2 
- L_{x,x+1}^{\mathrm{dKMP}} (\eta_x \eta_{x+1}) 
\le D_{\mathrm{dKMP}} (\eta_x^2 + \eta_{x+1}^2) 
\end{equation*}
where we used the trivial bound $\eta_x\le \eta_x^2$, which holds for particle-type configurations.

\subsection{Harmonic model}
Finally, for the Harmonic model $\eta=\eta^{\mathrm{Harm}}$ we can also compute directly the action
$L_{x,x+1}^{\mathrm{Harm}} (\eta_x\eta_{x+1})$ as follows. 
Note here that we have the following identities: 
\begin{equation}
\label{ident_1}
\sum_{k=1}^{n} \dfrac{\Gamma(n+1)\Gamma(a-k+n)}{\Gamma(a+n)\Gamma(n-k+1)} = \dfrac{n}{a}.  
\end{equation}
and 
\begin{equation}
\label{eq:harmonic_model_useful_identity2}
\sum_{k=1}^{n} \dfrac{k\Gamma(n+1)\Gamma(a-k+n)}{\Gamma(a+n)\Gamma(n-k+1)} = \dfrac{(a+n)n}{a(a+1)},
\end{equation} which hold for any $n\in\mathbb N$ and $a>0$. 
Indeed, the last identities can be shown via some elementary computation for the Beta-binomial distribution.
Here recall that probability mass function of the Beta-binomial distribution with parameters $\alpha_0,\beta_0>0$ is given by
\begin{equation*}
\binom{n}{k} \frac{B(k+\alpha_0, n-k+\beta_0)}{B(\alpha_0,\beta_0)}, 
\end{equation*}
for each $k=0,\ldots,n$ and that its mean is given by $n\alpha_0/(\alpha_0+\beta_0)$. 
Then, for \eqref{ident_1}, we note that   
\begin{equation*}
\begin{aligned}
\sum_{k=1}^{n} \dfrac{\Gamma(n+1)\Gamma(a-k+n)}{\Gamma(a+n)\Gamma(n-k+1)} 
&=(n+a)\sum_{k=1}^n \binom{n}{k} B(k+1,n-k+a)\\
&=(n+a)\Big( \sum_{k=0}^n \binom{n}{k} B(k+1,n-k+a)
- B(1,n+a) \Big)\\
&= (n+a)(B(1,a)- B(1,n+a) )\\
&= (n+a) \Big( \frac{1}{a} - \frac{1}{n+a} \Big) 
\end{aligned}
\end{equation*}
where we used the fact that the Beta-binomial distribution is a probability distribution. 
On the other hand, for \eqref{eq:harmonic_model_useful_identity2}, we note that 
\begin{equation*}
\begin{aligned}
\sum_{k=1}^{n} \dfrac{k\Gamma(n+1)\Gamma(a-k+n)}{\Gamma(a+n)\Gamma(n-k+1)} 
&=(n+a)\sum_{k=1}^n k\binom{n}{k} B(k+1,n-k+a)\\
&=(n+a)B(1,a) \frac{n}{a+1} 
= \frac{n(n+a)}{a(a+1)}. 
\end{aligned}
\end{equation*}
where we used the formula for the mean of the Beta-binomial distribution.  
Now, recalling $D_{\mathrm{Harm}}=1/(2\mathfrak s)$, we have that
\begin{equation*}
\begin{aligned}
&D_{\mathrm{Harm}} (\eta_x - \eta_{x+1})^2 
- L_{x,x+1}^{\mathrm{Harm}}(\eta_x \eta_{x+1}) \\
&\quad= \frac{1}{2\mathfrak s} (\eta_x - \eta_{x+1})^2 
- \sum_{k=1}^{\eta_x} \frac{\Gamma(\eta_{x} +1) \Gamma(\eta_{x} -k +2\mathfrak s)}{\Gamma(\eta_{x}  +2\mathfrak s)\Gamma(\eta_{x}-k+1)} 
(\eta_x - \eta_{x+1} -k)  \\
&\qquad- \sum_{k=1}^{\eta_{x+1}} \dfrac{\Gamma(\eta_{x} +1) \Gamma(\eta_{x+1} -k +2\mathfrak s)}{\Gamma(\eta_{x+1}  +2\mathfrak s)\Gamma(\eta_{x+1} -k+1)} 
(\eta_{x+1} - \eta_{x} -k) \\
&\quad= \dfrac{1}{2\mathfrak s(2\mathfrak s+1)} ( \eta_x^2 + \eta_{x+1}^2 + 2\mathfrak s \eta_x + 2\mathfrak s\eta_{x+1})
\le \frac{1}{2\mathfrak s}(\eta_x^2 + \eta_{x+1}^2) \;.
\end{aligned}
\end{equation*}
The last estimate follows from the trivial bound for particle-type occupation variables: $\eta_x\le \eta_x^2$. 
Thus, as in the other cases, we have the desired bound
\begin{equation*}
D_{\mathrm{Harm}} (\eta_x - \eta_{x+1})^2 
- L_{x,x+1}^{\mathrm{Harm}} (\eta_x \eta_{x+1}) 
\le D_{\mathrm{Harm}} (\eta_x^2 + \eta_{x+1}^2) .
\end{equation*}

\section{Absolute continuity}

\begin{prop}\label{prop:abs_cont}
    Suppose $\pi \in \mcb{M}_+$ satisfies $|\langle \pi,G \rangle| \le \hat\rho\|G\|_{L^1(\mathbb 
T)}$ for any $G \in C(\mathbb T)$. Then $\pi$ is absolutely continuous with respect to the Lebesgue measure.
\end{prop}
\begin{proof}
To show the absolute continuity of the measure $\pi$, it is enough to show the following: for any $\varepsilon>0$, there exists some $\delta>0$ such that for any family of disjoint open intervals $\{(a_k,b_k)\}_{k=1,\ldots,m}$ we have $\pi(A)<\varepsilon$ provided $|A|< \delta$, where we set $A=\bigcup_{k=1}^m (a_k,b_k)$ and $|\cdot|$ stands for the Lebesgue measure.   
In what follows, let $F$ be the closure of the open set $A$ and let us take an approximation of $F$ by closed sets $\{F_n\}_{n\in\mathbb N} \subset \mathbb T$ at level $n\in\mathbb N$, i.e., 
\begin{equation*}
F_n = \{ x\in\mathbb T : \inf_{y\in F}|x-y| \le 1/n\}.    
\end{equation*}
Then, let $g_n:\mathbb T\to[0,1]$ be a continuous such that $g_n(x)=1$ when $x\in F_n$, whereas $g_n(x)=0$ for $x\in F_n^c$. 
Notice that 
\begin{equation*}
\begin{aligned}
\pi(A) 
\le \pi(F)
=\langle\pi,\mathbf{1}_F\rangle
\le \langle \pi, g_n \rangle
\le \hat\rho \| g_n\|_{L^1(\mathbb T)}
\le \hat\rho|F_n|
\end{aligned}
\end{equation*}
where we used the assumption in the second estimate.  
Now, noting $F_n \searrow F$, we have that the utmost left-hand side of the last display is in fact bounded from above by $\hat\rho|F|=\hat\rho|A|$, and thus we conclude the proof. 
\end{proof}

\subsection*{Acknowledgments}
P.G. and C.F. thank the hospitality of the University of Tokyo for their research visit during the period October-November 2024 when part of this work was developed. P.G. expresses warm thanks to Funda\c c\~ao para a Ci\^encia e Tecnologia FCT/Portugal for financial support through the projects UIDB/04459/2020, UIDP/04459/2020 and ERC/FCT. M.S. is supported by KAKENHI 24K21515. C.F. acknowledges financial support by the Istituto Nazionale di Alta Matematica through the program “Borse di studio per l'estero”.

\subsection*{Data Availability}
No datasets were generated or analyzed during the current study. 

\subsection*{Conflict of Interests}
The authors declare that they have no known competing financial interests or personal relationships that could have appeared to influence the work reported in this paper.

\bibliographystyle{abbrv}
\bibliography{ref}

\begin{thebibliography}{10}

\bibitem{billingsley2013convergence}
P.~Billingsley.
\newblock {\em Convergence of probability measures}.
\newblock John Wiley \& Sons, 2013.

\bibitem{capanna2024class}
M.~Capanna, D.~Gabrielli, and D.~Tsagkarogiannis.
\newblock On a class of solvable stationary non equilibrium states for mass exchange models.
\newblock {\em Journal of Statistical Physics}, 191(2):25, 2024.

\bibitem{carinci2013duality}
G.~Carinci, C.~Giardin{\`a}, C.~Giberti, and F.~Redig.
\newblock Duality for stochastic models of transport.
\newblock {\em Journal of Statistical Physics}, 152:657--697, 2013.

\bibitem{fajfrova2016invariant}
L.~Fajfrov{\'a}, T.~Gobron, and E.~Saada.
\newblock Invariant measures of mass migration processes.
\newblock {\em Electronic Journal of Probability}, 21(60):1--52, 2016.

\bibitem{feng1997microscopic}
S.~Feng, I.~Iscoe, and T.~Sepp{\"a}l{\"a}inen.
\newblock A microscopic mechanism for the porous medium equation.
\newblock {\em Stochastic Processes and their Applications}, 66(2):147--182, 1997.

\bibitem{franceschini2023integrable}
C.~Franceschini, R.~Frassek, and C.~Giardin{\`a}.
\newblock Integrable heat conduction model.
\newblock {\em Journal of Mathematical Physics}, 64(4), 2023.

\bibitem{franceschini2022symmetric}
C.~Franceschini, P.~Gon{\c{c}}alves, and F.~Sau.
\newblock Symmetric inclusion process with slow boundary: Hydrodynamics and hydrostatics.
\newblock {\em Bernoulli}, 28(2):1340--1381, 2022.

\bibitem{frassek2020non}
R.~Frassek, C.~Giardin{\`a}, and J.~Kurchan.
\newblock Non-compact quantum spin chains as integrable stochastic particle processes.
\newblock {\em Journal of Statistical Physics}, 180(1):135--171, 2020.

\bibitem{dualitybook}
C.~Giardin\`{a} and F.~Redig.
\newblock {\em Duality for {M}arkov processes: a {L}ie-algebraic approach}.
\newblock In preparation, 2025.

\bibitem{gobron2010couplings}
T.~Gobron and E.~Saada.
\newblock Couplings, attractiveness and hydrodynamics for conservative particle systems.
\newblock {\em Annales de l'IHP Probabilit{\'e}s et statistiques}, 46(4):1132--1177, 2010.

\bibitem{guo1988nonlinear}
M.~Z. Guo, G.~C. Papanicolaou, and S.~S. Varadhan.
\newblock Nonlinear diffusion limit for a system with nearest neighbor interactions.
\newblock {\em Communications in Mathematical Physics}, 118(1):31--59, 1988.

\bibitem{kipnis1999scaling}
C.~Kipnis and C.~Landim.
\newblock {\em Scaling limits of interacting particle systems}, volume 320.
\newblock Springer Science \& Business Media, 1999.

\bibitem{kipnis1982heat}
C.~Kipnis, C.~Marchioro, and E.~Presutti.
\newblock Heat flow in an exactly solvable model.
\newblock {\em Journal of Statistical Physics}, 27:65--74, 1982.

\bibitem{liggett1985interacting}
T.~M. Liggett.
\newblock {\em Interacting particle systems}.
\newblock Springer, 1985.

\bibitem{lipatov1993high}
L.~Lipatov.
\newblock High energy asymptotics of multi--colour {QCD} and exactly solvable lattice models.
\newblock {\em Journal of Experimental and Theoretical Physics Letters}, 59:596--599, 1994.

\bibitem{rezakhanlou1991hydrodynamic}
F.~Rezakhanlou.
\newblock Hydrodynamic limit for attractive particle systems on {$\mathbb Z^d$}.
\newblock {\em Communications in mathematical physics}, 140(3):417--448, 1991.

\bibitem{suzuki1993hydrodynamic}
Y.~Suzuki and K.~Uchiyama.
\newblock Hydrodynamic limit for a spin system on a multidimensional lattice.
\newblock {\em Probability theory and related fields}, 95:47--74, 1993.

\end{thebibliography}

\end{document}